\documentclass[reqno,twoside,times, 10pt]{amsart}

\usepackage{hyperref}
\usepackage{enumerate,fullpage}
\usepackage{amssymb,amsmath,graphicx,amsthm}
\usepackage{color}
\allowdisplaybreaks

\newcommand{\ba}{\begin{eqnarray}}
\newcommand{\ea}{\end{eqnarray}}

\usepackage{amsfonts,amsmath,amsthm,mathrsfs}
\usepackage{amssymb}
\usepackage{color}
\usepackage{enumerate}
\usepackage{subfloat,subfigure}
\usepackage{graphics,verbatim,url}
\usepackage{index}

\newcommand{\A}{\mathcal{A}}

\newcommand{\be}{\begin{eqnarray*}}
	\newcommand{\en}{\end{eqnarray*}}
\newcommand{\bes}{\begin{eqnarray}}
	\newcommand{\ens}{\end{eqnarray}}

\def\leqslant {\le}
\usepackage{epic, curves}
\def\nn{\nonumber}

\newcommand{\al}{\alpha}

\newcommand{\ep}{\epsilon}
\newcommand{\bn}{ \beta_{\bf n} }

\newtheorem{theorem}{Theorem}[section]

\newtheorem{corollary}{Corollary}[section]
\newtheorem{definition}{Definition}[section]
\newtheorem{lemma}{Lemma}[section]
\newtheorem{proposition}{Proposition}[section]
\newtheorem{example}[theorem]{Example}

\newtheorem{remark}{Remark}[section]

\def\leqslant {\le}
\def\bq{\begin{equation}}
	\def\eq{\end{equation}}
\def\bqq{\begin{eqnarray*}}
	\def\eqq{\end{eqnarray*}}

\def\nn{\nonumber}

\begin{document}
	\title{Regularized solutions for  some  backward nonlinear parabolic equations  with  statistical data }         % Enter your title between curly braces
	\author{Mokhtar Kirane, Erkan Nane  and Nguyen Huy Tuan}          % Enter your name between curly braces
	\address{M. Kirane}
	\address{LaSIE, Facult\'{e} des Sciences et Technologies, Universi\'{e} de La Rochelle,  Avenue M. Cr\'{e}peau, 17042 La Rochelle, France}
	\email{mokhtar.kirane@univ-lr.fr }
	\address{E. Nane}
	\address{Department of Mathematics and Statistics, Auburn University, Auburn, USA }
	\email{ezn0001@auburn.edu }
	\address{N.H. Tuan}
	\address{Department of Mathematics and Informatics,  University of Science, Viet Nam National University VNU-HCMC, Viet Nam }
	\email{nguyenhuytuan@tdt.edu.vn or thnguyen2683@gmail.com  }
\maketitle

		\begin{abstract}
			In this paper, we  study  the backward problem of determining initial condition for some class of nonlinear parabolic equations in multidimensional domain where data are given under random noise.
			This problem is ill-posed, i.e., the solution  does not depend continuously on the data.   To regularize the instable solution, we develop some new methods to construct some new regularized solution.
			We also investigate the convergence
			rate between the regularized solution and the  solution of our equations. In particular, we establish results for  several equations with constant coefficients and time dependent coefficients. The equations with constant coefficients include heat equation,  extended Fisher-Kolmogorov equation,  Swift-Hohenberg equation and many others.  The equations with time dependent coefficients include   { Fisher type Logistic equations},   { Huxley equation},  { Fitzhugh-Nagumo equation}.  The methods developed in this paper can  also be applied to get approximate solutions to several other equations including 1-D 	{ Kuramoto-Sivashinsky	 equation},   1-D 	{ modified Swift-Hohenberg  equation}, strongly damped wave equation and 1-D Burger's equation with randomly perturbed operator.
		\end{abstract}

%		\tableofcontents
		
		\section{Introduction}

		In this paper,  we focus on the problem  of finding  the initial  functions ${\bf u }({\bf x},0)=u_0({\bf x})$ such that ${\bf u}$  satisfies the following nonlinear parabolic equation
		\bq
		\left\{ \begin{gathered}
			{\bf u}_t+\A(t,{\bf u}){\bf u}  = F({\bf u}({\bf x},t))+G({\bf x},t) ,~~0<t<T,x\in \Omega, \hfill \\
			{\bf u}({\bf x},t)= 0,~~{\bf x} \in \partial \Omega,\hfill\\
			{\bf u}({\bf x},T)= H({\bf x}), {\bf x} \in \Omega \hfill\\
		\end{gathered}  \right. \label{problem2}
		\eq
		where  the domain $\Omega=(0,\pi)^d$ is a   subset of  $\mathbb{R}^d$ and { ${\bf x}:=(x_1,...x_d)$.}    The functions $F$ and $G$ are called the source functions that satisfy the usual Lipschitz and growth conditions. The function $H$ is given and  is often called   a final value data. The operator $\A$  is  given by  the Laplacian, or a function of the Laplacian defined by the spectral theorem.

		The problem \eqref{problem2}   is   called the backward problem for classical parabolic equation when $\mathcal{A}=-\Delta$.
		It is applied in fields as the heat conduction theory \cite{BBC}, hydrology \cite{JB}, groundwater contamination \cite{SK}, to digitally remove blurred noiseless image \cite{CSH} and also in many other practical applications of mathematical physics and engineering \cite{zuazua2}.  It is well-known that  the backward parabolic  problem is severely ill--posed in the sense of Hadamard \cite{hada}. Hence  solutions do not always exist, and in the case of existence,
		the solutions  do not  depend continuously on the initial  data. In fact, from small noise contaminated physical measurements,
		the corresponding solutions might have large errors. More details on ill-posedness of the problem   with random noise can be found in \cite{Minh}.
		
The analysis
		of regularization methods for the stable solution (in the sense of Hadamard) of problem \eqref{problem2} depends on the mathematical
		model with  the noise term on the source function $G$ and  the final value data ${\bf u}_T={\bf u}({\bf x}, T)=H(x)$.  We
		suppose that the measurements are described as follows
		\begin{equation}
			G^{\text{obs}}= G+ \text{"noise"}, \quad H^{\text{obs}}= H+ \text{"noise"}. \label{noise}
		\end{equation}
		
		\subsection{Background in the  deterministic case}\label{subsection1.1}
		
		The Problem \eqref{problem2} is a generalized form of a class of backward parabolic equations. We give a short history of this problem in the   deterministic case.  If the "noise" (introduced in \eqref{noise}) is considered as a deterministic quantity, it is natural to study  what happens when  $\|\text{noise}\|_{L^2} \to 0$.
		
		When the operator $\A(t) = \A$ ( independent of  $t$),
		regularization  results were considered by many authors,   we refer the reader  to the survey  paper of
		Tuan \cite{Tuan} and the references therein.\\
		
		When $\A(t)$ depends on $t$ and $F=G=0$,  Lions and Lattes \cite{Lion}  proposed  the following quasi-reversibility method:
		\begin{equation}  \label{eq:lions-pb}
			\left\{
			\begin{array}{rcl}
				{\bf u}_\ep'(t) + \A(t) {\bf u}_\ep+\ep \A^{*}(t) \A(t)  {\bf u}_\ep  &=& 0 \\
				{\bf u}_\ep(T)&=&{\bf u}_T^\ep.
			\end{array}\right.
		\end{equation}
		However, Lions and Lattes did  not  study regularization results for this problem.
		\noindent The regularization result  here    is still open although some progress has
		been made.
		The first paper on this case
		seems to be that of Krein [12], where he used the log-convexity method to get
		stability estimates of H\"older type. His method and results have been further developed by
		Hao and Duc \cite{Hao}.\\
		
		When $\A(t)$ depends on $t$ and  ${\bf u}$,  to the best of our knowledge, there do not exist any results on  the backward problem.  Regularization results in here are very difficult because  one can not represent the  solution with a  nonlinear integral as previously done by others. Hence, the regularized solution can not be obtained  with the previous techniques based on nonlinear integral method. Regularization results for problem \eqref{problem2} in the deterministic case are still open.

		\subsection{Background on problem with random noise}

		If the errors are generated from uncontrollable sources as wind, rain, humidity, etc, then the model is random.
		If the "noise" (introduced in \eqref{noise})  are
		modeled as a random quantity, the convergence of estimators  $\widetilde {\bf u}({\bf x},0)$ of  ${\bf u}({\bf x},0)$ should be studied
		by  statistical methods.  Methods applied to  the deterministic cases cannot  be applied directly for this case.   The main idea in using  the random noise is of finding suitable  estimators
		$\widetilde {\bf u}({\bf x},0)$ and to consider the expected square error $\mathbb{E} \Big[\|\widetilde {\bf u}({\bf x},0) - {\bf u}({\bf x},0)\|^2\Big]$ in a suitable space, also  called the mean integrated square error (MISE).
		
		There exist a considerable amount of literature on regularization methods for
		linear backward problem  with random noise.
		When   $F({\bf u})=0$, the problem \eqref{problem2} is  linear and its solution  can be defined by a linear operator with random noise
		\begin{equation}
			u_T=K u_0 + \text{"noise"}.  \label{K}
		\end{equation}
		where $K$ is a bounded linear operator that  does not have a continuous inverse.
		There are  many well-known methods  including spectral cut-off (or called truncation method) of Cavalier \cite{Bi2,Cavalier},  the Tiknonov method \cite{Cox}, iterative regularization methods \cite{Engl}.
		Mair  and Ruymgaart  \cite{Mair} considered theoretical formulae  for statistical inverse estimation in Hilbert spaces
		and applied the method to solve the backward heat problem. Recently, Hohage et al. \cite{Hohage} applied  spectral cut-off (truncation method) and Tikhonov-type method to solve linear statistical inverse problems including backward heat equation (See p. 2625, \cite{Hohage}).  Recently, Problem \eqref{problem2} in the case of $F=0$ has been   studied in \cite{Minh} in  the plane domain.

		Until now, to the best of the authors' knowledge,  there
		are only a few  results in the case of random source, or random final value observations for nonlinear backward parabolic equation. And there are no results on Problem \eqref{problem2}. This is  our motivation in the  present paper.  In a few sentences,
		we give explanation  why the nonlinear problem is  difficult to investigate. Indeed, when   $F$ depends on ${\bf u}$, we can not transform the solution of problem \eqref{problem2} into \eqref{K}, this makes the nonlinear problem more challenging. Furthermore, as introduced in the subsection \ref{subsection1.1}, if $\A=\A(t,{\bf u})$ then we can not transform the problem \eqref{problem2} into a nonlinear integral equation, then the previous methods can not be applied for regularizing the problem.
		{\bf So, our task in this paper is   developing and establishing  new methods for solving this  problem.}
		
		\subsection{Outline  of the article.}
		
		In this paper, inspired by the random model in  \cite{Minh}, we  introduce the following   random model in  $\mathbb{R}^d$.
		
		Let $\Omega=(0,\pi)^d \subset \mathbb{R}^d$ for $d\geq 1$.  Let us recall the functions $H$ and $G$ from equation  \eqref{problem2}.  Let ${\bf x}_{\bf i }=( x_{i_1},...x_{i_d})$ be  grid points of $\Omega$ with index  ${\bf i}=(i_1,i_2,...i_d) \in \mathbb{N}^d, 1 \le i_k \le n_k $ for $k=\overline{1,d}$ where
		\begin{equation} \label{model1}
			{\mathbf x}_{\bf i }=( x_{i_1},...x_{i_d})=\Big(\frac{\pi(2i_1-1)}{2n_1}, \frac{\pi(2i_2-1)}{2n_2},...\frac{\pi(2i_d-1)}{2n_d}  \Big),\quad i_k= \overline {1,n_k},\quad k=\overline {1,d}.
		\end{equation}
		We consider the following   nonparametric regression model of  data as follows
		\begin{align} \label{model2}
			\widetilde D_{\bf i}=\widetilde D_{i_1,i_2,...i_d}:&= H	( x_{i_1},...x_{i_d})+ \Lambda_{i_1,i_2,...i_d} \Upsilon_{i_1,i_2,...i_d}=H({\bf x}_{\bf i })+\Lambda_{\bf i}\Upsilon_{\bf i}  \\
			\widetilde G_{\bf i}(t)=\widetilde G_{i_1,i_2,...i_d}(t):&=G	( x_{i_1},...x_{i_d},t)+\vartheta  \Psi_{i_1,i_2,...i_d}(t)=G({\bf x}_{\bf i },t)+\vartheta  \Psi_{\bf i}(t),
		\end{align}
		for $ i_k= \overline {1,n_k},\quad k=\overline {1,d}$.   Here $\Upsilon_{\bf i}:=\Upsilon_{i_1,i_2,...i_d} \sim \mathcal{N} (0,1)$  and $\Psi_{\bf i}(t):=\Psi_{i_1,i_2,...i_d}(t)$ are Brownian motions. We assume furthermore that they are  mutually independent.
		
		Our main goal in this paper is to provide some regularized solutions that are called estimators  for approximating ${\bf u}({\bf x},t),~0\le t <T$.
{  In this paper, we do not  investigate the existence and uniqueness of the solution of backward problem \eqref{problem2}.
			The uniqueness of backward parabolic has attracted the attention of many authors; see,  for example, \cite{Lu,Ru,Wu}.
			 It is also  a challenging  and   open problem, and  should be the  topic for another  paper.  In this paper, we assume that  the backward problem \eqref{problem2} has a unique solution ${\bf u}$ (called sought solution) { that  belongs to an appropriate space}.  So our main purpose  is to consider a regularized problem for finding  an approximate solution in the random cases}.   Furthermore,  error estimates with  the speed of convergence between the regularized solution and the sought solution  under some a priori assumptions on the sought solution are also our primary purpose.  { In particular,   the main purpose in  our error estimates is to show that  the norm of difference between the regularized solution and the sought solution in $L^2(\Omega)$  tends to zero when $ |{\bf n}|=\sqrt{n_1^2+...+n_d^2} \to +\infty$.}  { Our methods in this paper can be applied to solve  the backward problem in the  deterministic case that was   introduced in subsection \ref{subsection1.1}.}

		For the purpose of  capturing the main points of the paper, we    consider Problem \eqref{problem2}  and describe our main results    for the  three cases of $\A(t,{\bf u})$.\\
		
{\bf Case 1}:  $\A(t,{\bf u})=\A$.
		In this case,  we apply Fourier truncation method associated with knowledge on trigonometric theory in nonparametric regression for establishing regularized solutions.  The process for finding the regularized solution is given by the following steps :  First, we approximate the given data $H$ and $G$ by approximating functions $\widehat H_{\beta_{\bf n} }$ and $\widehat G_{\beta_{\bf n}}$ defined by Theorem \eqref{theorem2.1}.  Then,  we express the solution of Problem \eqref{problem2} into a nonlinear integral equation which is represented as Fourier series and then  we give some regularized solutions which are defined by other nonlinear integral equations. In this case, we will derive rates of convergence  under some a priori assumptions of the sought solution  ${\bf u}$. Main result in this case is given by Theorems \ref{theorem3.1}.  \\
		
{\bf Case 2}:   $\A(t)$ depends only  on $t$. As discussed before,  regularized methods  used in  Case 1 cannot  be  applied in  this case.  Hence, we  need  to figure out  a new regularized method to establish  a regularized solution.  Our main idea in this case is that of  applying  a modified  Quasi-reversibility method given by Lions\cite{Lion}. We will not  approximate directly the time dependents operator $\A(t)$ as introduced in \cite{Lion}. Our method is  of finding the unbounded time independent operator ${\bf P}$  that satisfies conditions in Definition  \ref{asumption1111}. Then, we  approximate   ${\bf P}$
		by a bounded operator ${\bf P_{\rho_n}}$,  in order to establish  the well-posedness of the problem associated with the  approximating functions $\widehat H_{\beta_{\bf n}}$ and $\widehat G_{\beta_{\bf n}}$.  Finding suitable regularized operators is important task in this section.  Our main results in this case are Theorem \ref{theorem4.1} and \ref{theorem4.2}. A special case of  Theorem \ref{theorem4.1} is  Corollary \ref{corollary4.1} which  studies the {\bf extended Fisher-Kolmogorov equation.}

% {\color{red} Please mention here the particular porblems/examples our results apply}\\
		
{\bf Case 3}:  $\A$  depends on $t$ and ${\bf u} $. Using the method in Case 2, we extend the results of Case 2  to the case when the coefficients of  $\A$  depends on $t$ and $u$.  Special cases of the equation considered in  Theorem \ref{theorem4.2} are {\bf Fisher type Logistic equations} with $F({\bf u})=a{\bf u}(1-{\bf u})$, or  {\bf Huxley equation} with $F({\bf u})=a{\bf u}^2(1-{\bf u})$, or {\bf Fitzhugh-Nagumo equation} with  $F({\bf u})=a{\bf u}^2(1-{\bf u})({\bf u}-\theta_1)$. See Remark \ref{remark4.2} for more.

		Finally, we want to mention that the  backward problem for some  concrete  nonlocal parabolic equations such as {\bf Ginzburg-Landau equation}  where  the coefficients of these equations are perturbed by random noises can be studied with the methods in our paper. These  equations include the 1-D {\bf Burger's equation.} Furthermore, our analysis and methods in this paper can be applied to get approximate solutions for many well-known equations. We state some examples below. We will work on these problems in a forthcoming paper.

			\begin{itemize}
				\item Backward problem for  1-D 	{\bf Kuramoto-Sivashinsky	 equation:}
					\begin{equation}
					\label{problem4444KS}
					{\bf u}_t +  d_0  {\bf u}_{xx}+ \left( d_1 (x){\bf u}_{xx} \right)_{xx}  =-{\bf uu}_x+G(x,t),  \qquad (x,t) \in  ( 0, \pi)\times (0,T),
					\end{equation}	
					with the conditions
					\begin{equation}
					\label{condition}
					\left\{\begin{array}{l l l}
					{\bf u}(0,t)&={\bf u}(\pi,t)={\bf u}_{xx}(0,t)={\bf u}_{xx}(\pi,t)= 0, & \qquad  t \in (0,T),\\
					{\bf u}(x,T) & = H(x), & \qquad x \in (0, \pi),
					\end{array}
					\right.
					\end{equation}
					where $d_0, d_1 $ are difussion  coefficients.
				 This
				 nonlinear partial differential equation describes incipient instabilities in a variety of physical and
				 chemical systems (see \cite{Cerpa,Cerpa2}).

	\item Backward problem for 1-D 	{\bf modified Swift-Hohenberg  equation:}
					\begin{equation}
					{\bf u}_t +  2  {\bf u}_{xx}+ k {\bf u}_{xxxx}  +a{\bf u}+ {\bf u}^3+ |{\bf u}_x|^2+ G(x,t)=0,  \qquad (x,t) \in  ( 0, \pi)\times (0,T),
					\end{equation}
					with condition \eqref{condition}.
					The Swift-Hohenberg equation is one of the universal equations used in the
					description of pattern formation in spatially extended dissipative systems, (see \cite{15}),
					which arise in the study of convective hydrodynamics \cite{16}, plasma confinement in
					toroidal devices \cite{5}, viscous film flow and bifurcating solutions of the Navier-Stokes
					\cite{12}.

				\item Backward problem for {\bf strongly damped wave equation:}
					\bq
					\left\{ \begin{gathered}
					{\bf u}_{tt} ({\bf x},t ) -\alpha  \Delta {\bf u}_t({\bf x}, t) -\Delta  {\bf u}({\bf x}, t) = F({\bf u}({\bf x},t))+G({\bf x},t) ,~~0<t<T, \ \ {\bf x} \in \Omega, \hfill \\
					{\bf u}({\bf x},t)= 0,~~{\bf x} \in \partial \Omega,\hfill\\
					{\bf u}({\bf x},T)= H({\bf x}), ~~~{\bf x} \in \Omega\hfill\\
					{\bf u}_t({\bf x},T)= 0,~~{\bf x} \in \Omega, \hfill\\
					\end{gathered}  \right. \label{problem77777}
					\eq
					where  $H  \in L^2(\Omega)$ is  a given function and $\al$ is a positive constant. 	Strongly damped wave equation,% (SDWE),
					occurs in a wide range of applications such as modeling motion of viscoelastic materials
					\cite{zuazua,massatt1983limiting,pata2005strongly}.  Some more physical applications of the equation \eqref{problem77777}  can be found in \cite{Pata}.

\item { \bf 1-D Burger's equation:}
						\begin{equation}
						\label{problem4444}
						\left\{\begin{array}{l l l}
						{\bf u}_t -(A(x,t)	{\bf u}_x)_x & =	{\bf u} 	{\bf u}_x+G(x,t), & \qquad (x,t) \in \Omega\times (0,T),\\
						{\bf u}(x,t)&=0, & \qquad x \in \partial {\Omega},\\
						{\bf u}(x,T) & = H(x), & \qquad x \in {\Omega},
						\end{array}
						\right.
						\end{equation}	
						where $\Omega =(0,\pi)$.
						The Burgers equation is a fundamental partial differential equation occurring in various areas of applied mathematics, such as fluid mechanics, nonlinear acoustics, gas dynamics, traffic flow \cite{Ku}. The ill-posedness of the backward problem for Burgers equation has been introduced  by  E. Zuazua et al \cite{zuazua} .
						The  model here is  as  follows:\\
						{\it
							Assume
							the time dependent coefficient $A(x,t)$ is noisy by random data
							\begin{equation}
							\widetilde A_k(t)= A(x_k,t)+  { \overline \vartheta} { \xi}_{k}(t),\quad \text{for} \quad  k=\overline{1,n}.  \label{noise1}
							\end{equation}
							where  $x_k =\frac{(2k-1)\pi}{2n},k=\overline{1,n} $ are the   grid points in $(0, \pi)$ and $\xi_{k}(t),k=\overline{1,n} $ are independent Brownian motions.}
			\end{itemize}

		\section{Constructing a function from discrete random data}
		\setcounter{equation}{0}
		In this section, we develop  a new theory for constructing a function in $L^2(\Omega)$ from the  given discrete random data.
		
		{	We first  introduce notation, and then we state the main results of this paper.

We will occasionally use the following  Gronwall's inequaly in this paper.
			\begin{lemma}
			Let $ b: [0,T] \to \mathbb{R}^+ $ be a continuous function  and $ C,D >0$ be constants that are  independent of $t$,  such that
			\begin{equation*}
			b(t) \le C+ D \int_{t}^T b(\tau)d\tau,\ \ \ t>0.
			\end{equation*}
			Then we have
			\begin{equation*}
			b(t) \le C e^{D(T-t)}.
			\end{equation*}
			\end{lemma}
			Next we define fractional powers of the Dirichlet Laplacian
			%\color{red}
			\begin{equation*}
				Af:= -\Delta f.
			\end{equation*}
			Since $A$ is a linear, densely defined self-adjoint and positive definite
			elliptic operator on the connected bounded domain  $\Omega $ with
			Dirichlet  boundary condition, {  using spectral theory, it is easy to show that  the eigenvalues of $A$ are given by}  $\lambda_{\bf p}=|{\bf p}|^2= p_1^2+p_2^2+\cdots +p_d^2  $.
			The corresponding eigenfunctions are denoted respectively by
			\begin{equation}\label{eigenfunctions-laplacian}
				\psi_{\bf p}({\bf x})=\bigg(\sqrt{\frac{2}{\pi}}\bigg)^d\sin (p_1 x_1)\sin(p_2x_2)\cdots \sin(p_d x_d).
			\end{equation}
			Thus the eigenpairs $(\lambda_{\bf p},\psi_{\bf p})$,
			$p\in \mathbb{N}^d$, satisfy
			\[
			\begin{cases}
			A \psi_{\bf p}({\bf x})
			=
			-\lambda_{\bf p} \psi_{\bf p}({\bf x}),
			\quad & {\bf x} \in \Omega \\
			\psi_{\bf p}({\bf x})
			=
			0,
			\quad & {\bf x}\in \partial \Omega.
			\end{cases}
			\]
			The functions $\psi_{\bf p}$ are normalized so that
			$\{\psi_{\bf p}\}_{{\bf p}\in \mathbb{N}^d}$ is an orthonormal basis of $L^2(\Omega)$.\\
		}
		We will use the following notation:
		$|{\bf p}| =|(p_1,\cdots, p_d)|= \sqrt{p_1^2+...+p_d^2}$, $|{\bf n}|=|(n_1,\cdots, n_d)|= \sqrt{n_1^2+...+n_d^2}$.
		
		\begin{definition}
			For $\gamma>0$, we define
			\begin{equation}
				\mathcal{H}^\gamma(\Omega):= \Big\{ h \in L^2(\Omega):  \sum_{p_1=1}^\infty...\sum_{p_d=1}^\infty |{\bf p}|^{2\gamma}  <h, \psi_{\bf p}>^2 ~ <\infty \Big\}.
			\end{equation}
			The norm on $\mathcal{H}^\gamma(\Omega)$ is defined by
			\begin{equation}
				\|h\|^2_{	\mathcal{H}^\gamma(\Omega)}:=	\sum_{p_1=1}^\infty...\sum_{p_d=1}^\infty |{\bf p}|^{2\gamma}  <h, \psi_{\bf p}>^2.	
			\end{equation}
			
		\end{definition}

		{   For any   Banach space $X$, we denote by $L_{p}\left(0,T;X\right)$,  the Banach space of  measurable real functions
			$v:(0,T)\to X$  such that
			\begin{align*}
				\left\Vert v\right\Vert _{L^{p}\left(0,T;X\right)}=\left(\int_{0}^{T}\left\Vert v\left(\cdot,t\right)\right\Vert _{X}^{p}dt\right)^{1/p}<\infty,\quad 1\le p<\infty,
			\end{align*}
			\begin{align*}
				\left\Vert v\right\Vert _{L^{\infty}\left(0,T;X\right)}= \text{esssup}_{0<t<T}\left\Vert v\left(\cdot,t\right)\right\Vert _{X}<\infty,\quad p=\infty.
			\end{align*}
			
		}		
	
		Let $\beta:\mathbb{N}^d\to \mathbb{R}$ be a function.
		Now we state our first main  result which gives  error estimate between $H$ and $\widehat H_{\beta_{\mathbf n}}$, and  error estimate between $\widehat G_{\beta_{\bf n}} $  and $G$.

			\begin{theorem}  \label{theorem2.1}
				Define the set $	\mathcal{W}_{   \beta_{\bf n}}$ for any ${\bf n}=(n_1,..n_d)\in \mathbb{N}^d$
				\begin{equation}
			\mathcal{W}_{  \beta_{\bf n}} = \Big\{ {\bf p} =(p_1,...p_d) \in \mathbb{N}^d  : |{\bf p}|^2= \sum_{k=1}^d p_k^2 \le \beta_{\bf n}= \beta(n_1,...n_d)  \Big\}
				\end{equation}
				where $\beta_{\bf n}$  satisfies
				\begin{equation*}
				\lim_{|{\bf n}| \to +\infty} \beta_{\bf n}=+\infty.
				\end{equation*}
				
				For a given ${\bf n}$ and $\beta_{\bf n}$ we define  functions that are approximating  $H, G$ as follows
				\begin{equation}
				\widehat H_{\beta_{\bf n}}({\bf x}) = \sum_{{\bf p} \in  	\mathcal{W}_{  \beta_{\bf n}}   } \Bigg[\frac{\pi^d}{\prod_{k=1}^d n_k } \sum_{i_1=1}^{n_1}...\sum_{i_d=1}^{n_d} \widetilde D_{i_1,i_2,...i_d} \psi_{\bf p}	( x_{i_1},...x_{i_d}) \Bigg] \psi_{\bf p}({\bf x})
				\end{equation}
				and
				\begin{equation}
				\widehat G_{\beta_{\bf n}}({\bf x},t) = \sum_{p \in 	\mathcal{W}_{  \beta_{\bf n}}  } \Bigg[\frac{\pi^d}{\prod_{k=1}^d n_k } \sum_{i_1=1}^{n_1}...\sum_{i_d=1}^{n_d} \widetilde G_{i_1,i_2,...i_d}(t) \psi_{\bf p}	( x_{i_1},...x_{i_d}) \Bigg] \psi_{\bf p}({\bf x}).
				\end{equation}
				Let $\mu=(\mu_1,...\mu_d) \in {\mathbb R}^d$ with $\mu_k >\frac{1}{2} $ for  any $ k=\overline {1,d}$. Let us choose $\mu_0 \ge { d \max(\mu_1,...\mu_d )}$.
				If  $H \in \mathcal{H}^{\mu_0}(\Omega) $ and $G \in L^\infty (0,T;\mathcal{H}^{\mu_0}(\Omega) )$ then the following estimates hold
				\begin{eqnarray}
				\begin{aligned}
				&	{\bf E}\Big\| \widehat H_{\beta_{\bf n}} -H  \Big\|^2_{L^2(\Omega)} \le  \overline C (\mu_1,...\mu_d, H) \beta_{\bf n}^{d/2} \prod_{k=1}^d  ( n_k)^{-4\mu_k}+{ 4\beta_{\bf n}^{-\mu_0}} \Big\| H \Big\|^2_{\mathcal{H}^{\mu_0}(\Omega) },\nn\\
				&	{\bf E}\Big\| \widehat G_{\beta_{\bf n}}(.,t) -G(.,t)  \Big\|^2_{L^\infty(0,T; L^2(\Omega))} \le  \overline C (\mu_1,...\mu_d, H) \beta_{\bf n}^{d/2} \prod_{k=1}^d  ( n_k)^{-4\mu_k}+ {4\beta_{\bf n}^{-\mu_0} }\Big\| G \Big\|^2_{L^\infty (0,T;\mathcal{H}^{\mu_0}(\Omega) ) }	,
				\end{aligned}
				\end{eqnarray}	
				where
				\begin{equation*}
				\overline C (\mu_1,...\mu_d, H)= 8\pi^d  V_{max}^2  \frac{2 \pi^{d/2}}{ d  \Gamma(d/2)} + \frac{ 16 \mathcal{C}^2 (\mu_1,...\mu_d) \pi^{d/2}}{ d  \Gamma(d/2)}  \Big\| H \Big\|^2_{\mathcal{H}^{\mu_0}(\Omega) }.
				\end{equation*}	
			\end{theorem}

		\begin{corollary}  \label{corollary2.1}
			Let  $H, G$ be as in  Theorem \eqref{theorem2.1}. Then  the term
			 $ {\bf E}	\Big\| \widehat H_{\beta_{\bf n}}-H \Big\|_{L^2(\Omega)}^2+T {\bf E}	\Big\| \widehat G_{\beta_{\bf n}}-G \Big\|_{L^\infty(0,T;L^2(\Omega))}^2$
			 is of order
			 \begin{equation*}  \label{veryimportant}
			 \max \Bigg( \frac{   \beta_{\bf n}^{d/2} }{\prod_{k=1}^d  ( n_k)^{4\mu_k} }, ~\beta_{\bf n}^{-\mu_0}  \Bigg).
			 \end{equation*}
		\end{corollary}
		To prove this Theorem, we need some  preliminary results.
		\begin{lemma}\label{lem2.1}
			Let ${\bf p,q} \in \mathbb{N}^d $ and ${\bf p}=(p_1,...p_d),~{\bf q}=(q_1,...q_d)$ with $p_k =\overline{1,n_k-1}$ and $x_{i_k}= \frac{\pi(2i_k-1)}{2n_k}$ for $k=\overline{1,d}$. Then for all ${\bf q} \in \mathbb{N}^d$, we have
			\begin{align}\label{tc1}
				\overline R_{\bf p,q}& =   \frac{1}{\prod_{k=1}^d n_k } \sum_{i_1=1}^{n_1}...\sum_{i_d=1}^{n_d} \psi_{\bf p}	( x_{i_1},...x_{i_d}) \psi_{\bf q}	( x_{i_1},...x_{i_d}) \nn\\
				&=  \left\{\begin{array}{*{20}{l}}
					\hfill \dfrac{1}{\pi^d}, & p_k \mp q_k=2l_k n_k ,~k=\overline {1,d},\\[8pt]
					\hfill 0, & \text{otherwise.}
				\end{array} \right.
			\end{align}
			If ${\bf q} \in \mathbb{N}^d$ satisfies that $ q_k := \overline{1,n_k-1}$ then we have
			\begin{equation}\label{tc1}
				\overline R_{\bf p,q} =   \frac{1}{\prod_{k=1}^d n_k } \sum_{i_1=1}^{n_1}...\sum_{i_d=1}^{n_d} \psi_{\bf p}	( x_{i_1},...x_{i_d}) \psi_{\bf q}	( x_{i_1},...x_{i_d}) =  \left\{\begin{array}{*{20}{l}}
					\hfill \dfrac{1}{\pi^d}, & p_k =q_k, ~k=\overline{1,d},\\
					\hfill 0, &\exists k=\overline{1,d}~p_k \neq q_k,
				\end{array} \right.
			\end{equation}
		\end{lemma}
		\begin{proof}
			The lemma is a direct consequence of Lemma 3.5 in \cite{Eubank}.
		\end{proof}

		\begin{lemma}  \label{lem2.3}
			Let  $H \in L^2 (\Omega)$ and  $H_{\bf p}=<H,\psi_{\bf p}>$ be   the Fourier coefficients of $H$ for ${\bf p}=(p_1,...p_d).$ Let's recall  that $x_{i_k}= \frac{2i_k-1}{2n_k},~k=\overline{1,d} $ then the following equality holds
			\begin{align}
				H_{\bf p}= \frac{\pi^d}{\prod_{k=1}^d n_k } \sum_{i_1=1}^{n_1}...\sum_{i_d=1}^{n_d} H	( x_{i_1},...x_{i_d}) \psi_{\bf p}	( x_{i_1},...x_{i_d})- \overline{\Gamma}_{{\bf n},{\bf p}}
			\end{align}
			where  ${\bf p}$  satisfies that $p_k=\overline{1,n_k-1},~k=\overline{1,d}$.
			Here
			$\overline{\Gamma}_{\bf n,{\bf p}}$ is defined by $$\overline{\Gamma}_{\bf n,{\bf p}} = \sum_{{\bf r}=2{\bf l}\cdot {\bf n} \pm {\bf p}}^{ {\bf l}\in \mathbb{N}^d, l_1^2+..l_d^2 \neq 0  }  H_{\bf r}. $$
		\end{lemma}
		\begin{proof}
			First, we have the expansion of the function $H \in L^2 (\Omega)$ as the following Fourier series
			\begin{equation}
				H({\bf x})= \sum_{ {\bf r} \in {\mathbb N}^d} H_{\bf r} \psi_{\bf r}({\bf x}).
			\end{equation}
			where ${\bf r}=(r_1,...r_d) \in {\mathbb N}^d $.
			Plug  ${\bf x}=	( x_{i_1},...x_{i_d}) $ into the latter equation, we obtain
			\begin{equation}
				H	( x_{i_1},...x_{i_d})= \sum_{ {\bf r} \in {\mathbb N}^d} H_{\bf r} \psi_{\bf r}( x_{i_1},...x_{i_d}).
			\end{equation}		
			This implies that
			\begin{eqnarray}
			\begin{aligned}
				&\frac{1}{\prod_{k=1}^d n_k } \sum_{i_1=1}^{n_1}...\sum_{i_d=1}^{n_d} H	( x_{i_1},...x_{i_d}) \psi_{\bf p}	( x_{i_1},...x_{i_d})\nn\\
				&=  \frac{1}{\prod_{k=1}^d n_k } \sum_{i_1=1}^{n_1}...\sum_{i_d=1}^{n_d}  \Bigg( \sum_{{\bf r} \in {\mathbb N}^d} H_{\bf r} \psi_{\bf r}	( x_{i_1},...x_{i_d})  \Bigg) \psi_{\bf p}	( x_{i_1},...x_{i_d})\nn\\
				&= \sum_{r \in {\mathbb N}^d} H_{\bf r} \Bigg( \frac{1}{\prod_{k=1}^d n_k } \sum_{i_1=1}^{n_1}...\sum_{i_d=1}^{n_d}  \psi_{\bf r}	( x_{i_1},...x_{i_d})  \psi_{\bf p}	( x_{i_1},...x_{i_d})   \Bigg)\nn\\
				&= \underbrace{\sum_{ {\bf r} \notin \mathcal{N}_d } H_{\bf r} \Bigg( \frac{1}{\prod_{k=1}^d n_k } \sum_{i_1=1}^{n_1}...\sum_{i_d=1}^{n_d}  \psi_{\bf r}	( x_{i_1},...x_{i_d})  \psi_{\bf p}	( x_{i_1},...x_{i_d})   \Bigg)}_{:=\mathbb{H}_1} \nn\\
				&+\underbrace{ \sum_{ {\bf r} \in \mathcal{N}_d} H_{\bf r} \Bigg( \frac{1}{\prod_{k=1}^d n_k } \sum_{i_1=1}^{n_1}...\sum_{i_d=1}^{n_d}  \psi_{\bf r}	( x_{i_1},...x_{i_d})  \psi_{\bf p}	( x_{i_1},...x_{i_d} \Bigg)}_{:=\mathbb{H}_2}  , \label{ss0}
			\end{aligned}
				\end{eqnarray}
			where we denote
			\begin{equation}
				\mathcal{N}_d:= \Big\{  {\bf r}= (r_1,r_2,...r_d)\in {\bf N}^d: r_k =\overline{1,n_k-1},~k=\overline{1,d}  \Big\}.
			\end{equation}
			{\bf Step 1.} Consider $\mathbb{H}_1$.\\
			If  $ {\bf r} \notin \mathcal{N}_d$ then	applying the first part of Lemma \ref{lem2.1}, we have
			\begin{align}\label{tc11}
				\frac{1}{\prod_{k=1}^d n_k } \sum_{i_1=1}^{n_1}...\sum_{i_d=1}^{n_d}  \psi_{\bf r}	( x_{i_1},...x_{i_d})  \psi_{\bf p}	( x_{i_1},...x_{i_d})
				&=  \left\{\begin{array}{*{20}{l}}
					\hfill \dfrac{1}{\pi^d}, & r_k \mp p_k=2l_k n_k ,~k=\overline {1,d},\\[8pt]
					\hfill 0, & \text{otherwise.}
				\end{array} \right.
			\end{align}
			Hence, we deduce that
			\begin{align}\label{tc1111}
				&\sum_{r \notin \mathcal{N}_d } H_{\bf r} \Bigg( \frac{1}{\prod_{k=1}^d n_k } \sum_{i_1=1}^{n_1}...\sum_{i_d=1}^{n_d}  \psi_{\bf r}	( x_{i_1},...x_{i_d})  \psi_{\bf p}	( x_{i_1},...x_{i_d})   \Bigg)\nn\\
				&\quad \quad =  \left\{\begin{array}{*{20}{l}}
					\hfill \dfrac{1}{\pi^d}\sum_{{\bf r} \notin \mathcal{N}_d } H_{\bf r}, & r_k \mp p_k=2l_k n_k ,~k=\overline {1,d},\\[8pt]
					\hfill 0, & \text{otherwise.}
				\end{array} \right.\nn\\
				&\quad \quad=  \dfrac{1}{\pi^d} \sum_{r \notin \mathcal{N}_d }^{ {\bf r}=2{\bf l}\cdot {\bf n} \pm {\bf p}}  H_{\bf r}.
			\end{align}
			where noting that the equation $r_k \mp p_k=2l_k n_k$ is equavilent to ${\bf r}=2{\bf l}\cdot {\bf n} \pm {\bf p}$.\\
			For  the sum $\sum_{r \notin \mathcal{N}_d }^{ {\bf r}=2{\bf l}\cdot {\bf n} \pm {\bf p}} H_r$
			on the right hand side of  \eqref{tc1111}, since  $ {\bf r }\notin \mathcal{N}_d$, we can see that  ${\bf r}$ is not different than ${\bf p}$.
			This implies  that
			${\bf l} =(l_1,...l_d) \in {\mathbb N}^d$ in the sum $\sum_{{\bf r} \notin \mathcal{N}_d }^{ {\bf r}=2{\bf l}\cdot {\bf n} \pm {\bf p}} H_{\bf r}$  satisfies the following condition
			\[
			l_1^2+...+l_d^2 \neq 0. \label{tc11111}
			\]
			Therefore, we can rewrite \eqref{tc1111}
			as follows
			\begin{align}
				\sum_{{\bf r} \notin \mathcal{N}_d } H_{\bf r} \Bigg( \frac{1}{\prod_{k=1}^d n_k } \sum_{i_1=1}^{n_1}...\sum_{i_d=1}^{n_d}  \psi_{\bf r}	( x_{i_1},...x_{i_d})  \psi_{\bf p}	( x_{i_1},...x_{i_d})   \Bigg)= \frac{1}{\pi^d}   \sum_{{\bf r}=2{\bf l}\cdot {\bf n} \pm {\bf p}}^{ {\bf l}\in \mathbb{N}^d, l_1^2+..l_d^2 \neq 0  }  H_{\bf r}.  \label{ss00}
			\end{align}
			{\bf Step 2.} Consider $\mathbb{H}_2$. \\
			If  $ {\bf r} \in \mathcal{N}_d$ then 	applying the second part of Lemma \ref{lem2.1}, we get
			\begin{align}\label{tc111}
				\frac{1}{\prod_{k=1}^d n_k } \sum_{i_1=1}^{n_1}...\sum_{i_d=1}^{n_d}  \psi_{\bf r}	( x_{i_1},...x_{i_d})  \psi_{\bf p}	( x_{i_1},...x_{i_d})
				&=  \left\{\begin{array}{*{20}{l}}
					\hfill \dfrac{1}{\pi^d}, & r_k =p_k, ~k=\overline{1,d},\\
					\hfill 0, &\exists k=\overline{1,d}~r_k \neq p_k,
				\end{array} \right.
			\end{align}
			This leads to
			\begin{align}
				\sum_{r \in \mathcal{N}_d } H_{\bf r} \Bigg( \frac{1}{\prod_{k=1}^d n_k } \sum_{i_1=1}^{n_1}...\sum_{i_d=1}^{n_d}  \psi_{\bf r}	( x_{i_1},...x_{i_d})  \psi_{\bf p}	( x_{i_1},...x_{i_d})   \Bigg)= \frac{1}{\pi^d} H_{\bf p}. \label{ss000}
			\end{align}
			Combining \eqref{ss0}, \eqref{ss00},\eqref{ss000}, we obtain
			\begin{align}
				\frac{1}{\prod_{k=1}^d n_k } \sum_{i_1=1}^{n_1}...\sum_{i_d=1}^{n_d} H	( x_{i_1},...x_{i_d}) \psi_{\bf p}	( x_{i_1},...x_{i_d})= \frac{1}{\pi^d}  \Bigg(  H_{\bf p}+ \sum_{{\bf r}=2{\bf l}\cdot {\bf n} \pm {\bf p}}^{ l\in \mathbb{N}^d, l_1^2+..l_d^2 \neq 0  }  H_{\bf r} \Bigg).
			\end{align}
			This completes the proof of this Lemma.
		\end{proof}
		
		\noindent 	Using Lemma \ref{lem2.3}, we obtain the following Lemma
		\begin{lemma} \label{lem2}
			Assume that $G \in C([0,T]; C^1(\overline{\Omega})) $ then for $t \in [0,T]$, we have
			\begin{equation}
				G_{\bf p}(t)= \frac{\pi^d}{\prod_{k=1}^d n_k } \sum_{i_1=1}^{n_1}...\sum_{i_d=1}^{n_d} G	( x_{i_1},...x_{i_d},t) \psi_{\bf p}	( x_{i_1},...x_{i_d})- \overline{\Pi}_{ \bf n,p}(t)
			\end{equation}
			where  ${\bf p}$  satisfies that $p_k=\overline{1,n_k-1},~k=\overline{1,d}$ and
			$$\overline{\Pi}_{\bf n,p} (t)= \sum_{{\bf r}=2{\bf l}\cdot {\bf n} \pm {\bf p}}^{ l\in \mathbb{N}^d, l_1^2+..l_d^2 =0  }  G_{\bf r}(t). $$
		\end{lemma}

		Next we consider the following Lemma
		\begin{lemma} \label{lem2.6}
			Assume that $\mu=(\mu_1,...\mu_d)$. Let us choose $\mu_0 \ge { d \max(\mu_1,...\mu_d )}$.
			Then if $H \in \mathcal{H}^{\mu_0}(\Omega) $ then
			\begin{align}
				|H_{\bf p}| =|H_{p_1,...p_d}|\le  d^{-\frac{  \max(\mu_1,...\mu_d) }{2}}  \Big\| H \Big\|_{\mathcal{H}^{\mu_0}(\Omega) } \prod_{k=1}^d p_k^{-\mu_k} , \quad \text{for}\quad {\bf p}=(p_1,p_2,...p_d).
			\end{align}
		\end{lemma}
		\begin{proof}
			Using the Cauchy inequality
			\begin{equation*}
				(a_1+...+a_d)^d \ge d \prod_{k=1}^d a_k
			\end{equation*}
			for any $a_k \ge 0, k=\overline{1,d}$, we have
			\begin{align*}
				\left( \sum_{k=1}^d p_k^2 \right)^{\frac{ d \max(\mu_1,...\mu_d) }{2}} &\ge \Big(d  \prod_{k=1}^d p_k^2  \Big)^{\frac{  \max(\mu_1,...\mu_d) }{2}}\nn\\
				&\ge d^{\frac{  \max(\mu_1,...\mu_d) }{2}} \prod_{k=1}^d p_k^{\mu_k}.
			\end{align*}
			The left hand side of the latter inequality is bounded by
			\begin{equation*}
				\left( \sum_{k=1}^d p_k^2 \right)^{\frac{ d \max(\mu_1,...\mu_d) }{2}} = |{\bf p}|^{   d \max(\mu_1,...\mu_d)  }  \le |{\bf p}|^{\mu_0}
			\end{equation*}
			The observations  above imply that
			\begin{equation*}
				\prod_{k=1}^d p_k^{\mu_k}  \le  d^{-\frac{  \max(\mu_1,...\mu_d) }{2}}  |{\bf p}|^{\mu_0}.
			\end{equation*}
			This leads to
			\begin{equation*}
				\prod_{k=1}^d p_k^{\mu_k}  \big|<H, \psi_{\bf p}>\big| \le d^{-\frac{  \max(\mu_1,...\mu_d) }{2}}  |{\bf p}|^{\mu_0} \big|<H, \psi_{\bf p}>\big|  \le d^{-\frac{  \max(\mu_1,...\mu_d) }{2}}  \Big\| H \Big\|_{\mathcal{H}^{\mu_0}(\Omega) }
			\end{equation*}
			where we used the fact that
			\begin{equation}
				\Big\| H \Big\|^2_{\mathcal{H}^{\mu_0}(\Omega) }=	\sum_{p_1=1}^\infty...\sum_{p_d=1}^\infty |{\bf p}|^{2\mu_0}  <H, \psi_{\bf p}>^2 ~\ge~ |{\bf p}|^{2\mu_0}  <H, \psi_{\bf p}>^2.	
			\end{equation}
			This completes the proof of Lemma.
		\end{proof}

		\begin{proof}[\bf Proof of Theorem \ref{theorem2.1}]

			\noindent 	Using Lemma \ref{lem2.3} and by a simple computation, we get
			\begin{align}
				&\Big\| \widehat H_{\beta_{\bf n}} -H  \Big\|^2_{L^2(\Omega)}\nn\\
				&= 4  \sum_{{\bf p} \in 	\mathcal{W}_{  \beta_{\bf n}} } \Bigg[\frac{\pi^d}{\prod_{k=1}^d n_k } \sum_{i_1=1}^{n_1}...\sum_{i_d=1}^{n_d} \Lambda_{i_1,i_2,...i_d} \Upsilon_{i_1,i_2,...i_d} \psi_{\bf p}	( x_{i_1},...x_{i_d}) -{\bf \overline \Gamma}_{\bf n,p}\Bigg]^2  \nn\\
				&\quad \quad \quad  + 4  \sum_{{\bf p} \notin  	\mathcal{W}_{  \beta_{\bf n}} } \Big| H_{\bf p}\Big|^2\nn\\
				&\le \underbrace{\frac{8\pi^{2d}}{\left(\prod_{k=1}^d n_k\right)^2 } \sum_{{\bf p} \in 	\mathcal{W}_{  \beta_{\bf n}} } \Bigg[\sum_{i_1=1}^{n_1}...\sum_{i_d=1}^{n_d} \Lambda_{i_1,i_2,...i_d} \Upsilon_{i_1,i_2,...i_d}  \Bigg]^2}_{:=A_{1,1}}\nn\\
				&+\underbrace{8 \sum_{{\bf p} \in 	\mathcal{W}_{  \beta_{\bf n}} } \Big|{\bf \overline \Gamma}_{\bf n,p} \Big|^2}_{:=A_{1,2}}+ \underbrace{4  \sum_{{\bf p} \notin  	\mathcal{W}_{  \beta_{\bf n}} } \Big| H_{\bf p}\Big|^2}_{:=A_{1,3}}  \label{errora}
			\end{align}
			where  we used the inequality $(a+b)^2 \le 2 a^2+ 2b^2$.
			The expectation of $A_{1,1}$ is bounded by
			\begin{align}
				{\bf  E} A_{1,1} & \le \frac{8\pi^{2d}}{\left(\prod_{k=1}^d n_k\right)^2 } \sum_{{\bf p} \in 	\mathcal{W}_{  \beta_{\bf n}} } \frac{  \prod_{k=1}^d n_k}{\pi^d} V_{\max}^2  \nn\\
				&= \frac{8 \pi^d}{ \prod_{k=1}^d n_k } V_{\max}^2  \text{card} \left(\mathcal{W}_{  \beta_{\bf n}} \right) ,
				 \label{errorb-1}
			\end{align}
			above we used the fact  that ${\bf  E}^2 \Upsilon_{i_1,i_2,...i_d}  =1$ and $ {\bf E} \left(\Upsilon_{i_1,i_2,...i_d} \Upsilon_{j_1,j_2,...j_d}  \right)  =0$.
			Now we estimate the
			\begin{equation}
			 \text{card} \left(\mathcal{W}_{  \beta_{\bf n}} \right)  = \text{card} \left(  \Big\{ {\bf p} =(p_1,...p_d) \in \mathbb{N}^d  :  \sum_{k=1}^d p_k^2 \le \beta_{\bf n}= \beta(n_1,...n_d)  \Big\} \right),
			 \end{equation}
		which is the number of ${\bf p}$ such that $ |{\bf p}|^2 \le \beta_{\bf n}$. Let any ${\bf p}=(p_1,...p_d) \in \mathbb{N}^d  $ such that $\sum_{k=1}^d p_k^2 \le \beta_{\bf n}$.
		  Let us define  rectangles  $ {\bf Q}_{ {\bf p}  }$   in $\mathbb{R}^d  $ as follows
		  \begin{equation*}
		  {\bf Q}_{ {\bf p}  } = \Big\{  {\bf z}= (z_1,z_2,...z_d) \in \mathbb{R}^d:  p_{k}-1 \le z_k \le p_k  \Big\}.
		  \end{equation*}
		Let us define  the set  $ {\bf Q}_{ \sqrt{\beta_{\bf n}}  }$   as follows
		\begin{align*}
		{\bf Q}_{ \sqrt{\beta_{\bf n}}  }= \bigcup_{ {\bf p}^2 \le \beta_{\bf n} }    {\bf Q}_{ {\bf p}  }.
		\end{align*}
		 It is easy to see that  $ \text{card} \left(\mathcal{W}_{  \beta_{\bf n}} \right)$ is equal to the  volume of the set  $ {\bf Q}_{ \sqrt{\beta_{\bf n}}  } $  and we can realize that
		 \begin{equation*}
		 {\bf Q}_{ \sqrt{\beta_{\bf n}}  } \subset {\bf Q'}_{ \sqrt{\beta_{\bf n}}  } = \Big\{ {\bf z} =(z_1,...z_d) \in \mathbb{R}^d  :  \sum_{k=1}^d z_k^2 \le \beta_{\bf n} \Big\}.
		 \end{equation*}
		  Hence $
		  \text{card} \left(\mathcal{W}_{  \beta_{\bf n}} \right)$ is less than the
		volume of the set   ${\bf Q'}_{ \sqrt{\beta_{\bf n}}  }$ which denoted by $ \text{Vol} \Big({\bf Q'}_{ \sqrt{\beta_{\bf n}}  }\Big)$ i.e,
		\begin{equation*}
		 \text{card} \left(\mathcal{W}_{  \beta_{\bf n}} \right) \le \text{Vol} \Big({\bf Q'}_{ \sqrt{\beta_{\bf n}}  }\Big). \label{ob0}
		\end{equation*}
		  Now we  find an upper bound of  $ \text{Vol} \Big({\bf Q'}_{ \sqrt{\beta_{\bf n}}  }\Big)$. First,  we have
		\begin{equation}
		\text{Vol} \Big({\bf Q'}_{ \sqrt{\beta_{\bf n}}  }\Big) = \int_{   {\bf Q'}_{ \sqrt{\beta_{\bf n}}  } } {\bf d}z_1 {\bf d}z_2...{\bf d}z_d.
		\end{equation}
		We give  the  following coordinate system as follows
	\begin{align*}
	&z_1= r \cos (\theta_1),~z_2=r  \sin (\theta_1)\cos (\theta_2), z_3= r \sin (\theta_1) \sin (\theta_2)\cos (\theta_3),...\nn\\
	& z_{d-1}= r \sin (\theta_1) ...\sin (\theta_{d-2})\cos (\theta_{d-1})\nn\\
	& z_{d}= r \sin (\theta_1) ...\sin (\theta_{d-2})\sin (\theta_{d-1})
	\end{align*}	
		where \[
		 1 \le r \le  \sqrt{\beta_{\bf n}} ,~ 0 \le \theta_i  \le \pi, i=\overline{1. d-2}, 0 \le \theta_{d-1} <2 \pi .
		 \]
	From the Change of Variables formula that the rectangular volume element
$ {\bf d}z_1 {\bf d}z_2...{\bf d}z_d$
	can be written in spherical coordinates as
	\begin{align*}
	 {\bf d}z_1 {\bf d}z_2...{\bf d}z_d &= \Bigg| \text{det} \Big( \frac{\partial x_i}{ \partial(r, \theta_j) }\Big)  \Bigg|_{1 \le i \le d, 1 \le j \le d-1}~~{\bf d}r  {\bf d}\theta_{1}...{\bf d} {\theta}_{d-1}\nn\\
	 &= r^{d-1} \sin^{d-2}(\theta_1) \sin^{d-3}(\theta_2) ...\sin (\theta_{d-2})  {\bf d}r  {\bf d}\theta_{1}...{\bf d} {\theta}_{d-1}.
	\end{align*}
	Hence, applying  Fubini's theorem, we obtain that
	\begin{align}
		&\text{Vol} \Big({\bf Q'}_{ \sqrt{\beta_{\bf n}}  }\Big)\nn\\
		& = \int_{   \mathcal{W}_{  \rho_{\bf n}}({\bf n}) } {\bf d}z_1 {\bf d}z_2...{\bf d}z_d\nn\\
		&= \int_0^{2\pi} \int_0^\pi ...\int_0^\pi \int_{1}^{ \sqrt{\beta_{\bf n} }} r^{d-1} \sin^{d-2}(\theta_1) \sin^{d-3}(\theta_2) ...\sin (\theta_{d-2})  {\bf d}r  {\bf d}\theta_{1}... {\bf d}\theta_{d-2} {\bf d} {\theta}_{d-1}\nn\\
		&= \left( \int_0^{2\pi}  {\bf d} {\theta}_{d-1} \right) \left( \int_{1}^{ \sqrt{\beta_{\bf n}} } r^{d-1} dr \right) \left( \int_0^{\pi}  \sin^{d-2}(\theta_1)  {\bf d}\theta_{1} \right)  \left( \int_0^{\pi}  \sin^{d-3}(\theta_2)  {\bf d}\theta_{2} \right)...\left( \int_0^{\pi}  \sin(\theta_{d-2})  {\bf d}\theta_{d-2} \right)\nn\\
		&= \frac{2\pi \left( \beta_{\bf n}^{d/2}-1 \right)}{d} \left( \int_0^{\pi}  \sin^{d-2}(\theta_1)  {\bf d}\theta_{1} \right)  \left( \int_0^{\pi}  \sin^{d-3}(\theta_2)  {\bf d}\theta_{2} \right)...\left( \int_0^{\pi}  \sin(\theta_{d-2})  {\bf d}\theta_{d-2} \right). \label{ob1}
	\end{align}
	Thanks to page 245, \cite{Giple}, we know that
	\begin{align}
	\left( \int_0^{\pi}  \sin^{d-2}(\theta_1)  {\bf d}\theta_{1} \right)  \left( \int_0^{\pi}  \sin^{d-3}(\theta_2)  {\bf d}\theta_{2} \right)...\left( \int_0^{\pi}  \sin(\theta_{d-2})  {\bf d}\theta_{d-2} \right)= \frac{d}{ 2\pi}  \text{Vol} \Big( \mathbb{B}^d(1)\Big).  \label{ob2}
	\end{align}
	Here $ \text{Vol} \Big( \mathbb{B}^d(1)\Big)$ denotes the volume of  of the unit
	$d$-ball.
	 Using again \cite{Giple} ( Proposition 4.2, page 246-247), we obtain that
	 \begin{equation}
	  \text{Vol} \Big( \mathbb{B}^d(1)\Big)= \frac  {  \pi^{d/2}}  {\Gamma(\frac{d}{2}+1)} . \label{ob3}
	 \end{equation}
	 Combining \eqref{ob1}, \eqref{ob2}, \eqref{ob3} gives
	 \begin{align}
	 \text{Vol} \Big({\bf Q'}_{ \sqrt{\beta_{\bf n}}  }\Big) \le \frac  {  \pi^{d/2} \beta_{\bf n}^{d/2} }  {\Gamma(\frac{d}{2}+1)} = \frac{2 \pi^{d/2}}{ d  \Gamma(d/2)} \beta_{\bf n}^{d/2}
	 \end{align}
	 which we have used the fact that $ \Gamma(\frac{d}{2}+1)= \frac{d}{2}\Gamma(\frac{d}{2}) $. This together with \eqref{ob0} leads to
			\begin{align}
				\text{card} \left(  	\mathcal{W}_{  \beta_{\bf n}} \right) = &\text{card} \left(  \Big\{ {\bf p} =(p_1,...p_d) \in \mathbb{N}^d,~p_k \ge 1, k=\overline{1,d}  :  \sum_{k=1}^d p_k^2 \le \beta_{\bf n}= \beta(n_1,...n_d)  \Big\} \right) \nn\\
				&\le \frac{2 \pi^{d/2}}{ d  \Gamma(d/2)} \beta_{\bf n}^{d/2}.  \label{card}
			\end{align}
			This implies that
			\begin{align}
				{\bf  E} A_{1,1}  \le   8\pi^d  V_{max}^2  \frac{2 \pi^{d/2}}{ d \Gamma(d/2)}   \frac{  \beta_{\bf n}^{d/2} }{ \prod_{k=1}^d n_k }. \label{errorb}
			\end{align}
			Next, in order to  estimate $A_{1,2}$, we need to find an  upper bound of ${\bf \overline \Gamma }_{\bf n,p}$.
			
			\noindent Using  Lemma \ref{lem2.6},  we  estimate $\Big|{\bf \overline \Gamma }_{\bf n,p}\Big|$ as follows
			\begin{align}
				\Big|{\bf \overline \Gamma }_{\bf n,p}\Big|  &\le \sum_{{\bf r}=2{\bf l}\cdot {\bf n} \pm {\bf p}}^{ l\in \mathbb{N}^d, l_1^2+..l_d^2 \neq 0  }  \Big|H_r\Big| = \sum_{ l\in \mathbb{N}^d, l_1^2+..l_d^2 \neq 0  }  \Big| H_{2l_1n_1 \pm p_1,...2l_dn_d \pm p_d} \Big|  \nn\\
				&\le d^{-\frac{  \max(\mu_1,...\mu_d) }{2}}  \Big\| H \Big\|_{\mathcal{H}^{\mu_0}(\Omega) }  \sum_{ l\in \mathbb{N}^d, l_1^2+..l_d^2 \neq 0  }  \prod_{k=1}^d (2l_k n_k \pm p_k)^{-\mu_k}  .   \label{es00000}
			\end{align}
			Furthermore, for $p_k =\overline {1,n_k}$ for $k=\overline {1,d}$,  we have
			\begin{align}
				\prod_{k=1}^d (2l_k n_k +  p_k)^{\mu_k} \ge \prod_{k=1}^d  ( n_k)^{\mu_k}  \prod_{k=1}^d  ( 2 l_k)^{\mu_k}  \label{es11111}
			\end{align}
			and
			\begin{align}
				\prod_{k=1}^d (2l_k n_k - p_k)^{\mu_k} \ge \prod_{k=1}^d (2l_k n_k - n_k)^{\mu_k} \ge \prod_{k=1}^d  ( n_k)^{\mu_k}  \prod_{k=1}^d  ( 2 l_k-1)^{\mu_k}.  \label{es22222}
			\end{align}
			Combining \eqref{es11111} and \eqref{es22222} gives
			\begin{align}
				\prod_{k=1}^d (2l_k n_k \pm p_k)^{-\mu_k} \le \prod_{k=1}^d  ( n_k)^{-2\mu_k}  \prod_{k=1}^d  ( 2 l_k-1)^{-2\mu_k}.
			\end{align}
			This together with \eqref{es00000} implies that
			\begin{align}
				\Big|{\bf \overline \Gamma }_{\bf n,p}\Big| \le d^{-\frac{  \max(\mu_1,...\mu_d) }{2}}  \Big\| H \Big\|_{\mathcal{H}^{\mu_0}(\Omega) } \prod_{k=1}^d  ( n_k)^{-2\mu_k} \sum_{ l\in \mathbb{N}^d, l_1^2+..l_d^2 \neq 0  }  \prod_{k=1}^d  ( 2 l_k-1)^{-2\mu_k}.
			\end{align}
			Since $\mu_k >\frac{1}{2}$, we know that the series  $  d^{-\frac{  \max(\mu_1,...\mu_d) }{2}}  \sum_{ l\in \mathbb{N}^d, l_1^2+..l_d^2 \neq 0  }  \prod_{k=1}^d  ( 2 l_k-1)^{-2\mu_k}$ converges. So we define this sum to be $\mathcal{C} (\mu_1,...\mu_d)$. Hence
			\begin{align}
				\Big|{\bf \overline \Gamma }_{\bf n,p}\Big| \le \mathcal{C} (\mu_1,...\mu_d) \Big\| H \Big\|_{\mathcal{H}^{\mu_0}(\Omega) } \prod_{k=1}^d  ( n_k)^{-2\mu_k}.
			\end{align}
			It  follows from \eqref{card} that
			\begin{align}
				A_{1,2}&= 8 \sum_{{\bf p} \in \mathcal{W}_{  \beta_{\bf n}} } \Big|{\bf \overline \Gamma}_{\bf n,p} \Big|^2 \nn\\
				&\le 8 \mathcal{C}^2 (\mu_1,...\mu_d) \Big\| H \Big\|^2_{\mathcal{H}^{\mu_0}(\Omega) } \prod_{k=1}^d  ( n_k)^{-4\mu_k} \text{card} \left(\mathcal{W}_{  \beta_{\bf n}} \right)\nn\\
				&\le \frac{ 16 \mathcal{C}^2 (\mu_1,...\mu_d) \pi^{d/2}}{ d  \Gamma(d/2)}  \Big\| H \Big\|^2_{ \mathcal{H}^{\mu_0}(\Omega) } \beta_{\bf n}^{d/2} \prod_{k=1}^d  ( n_k)^{-4\mu_k}.  \label{errorc}
			\end{align}
		For $A_{1,3}$ on the right hand side of \eqref{errora}, noting that $ {\bf |p|^2}  \ge \beta_{\bf n} $ if $ {\bf p} \notin  \mathcal{W}_{  \beta_{\bf n}} $,  we have the following estimate
		
\begin{equation}
		A_{1,3}=4  \sum_{{\bf p} \notin  \mathcal{W}_{  \beta_{\bf n}}   }  |{\bf p}|^{-2\mu_0} |{\bf p}|^{2\mu_0} \Big| H_{\bf p}\Big|^2 \le { 4 \beta_{\bf n}^{-\mu_0}} \Big\| H \Big\|^2_{\mathcal{H}^{\mu_0}(\Omega) }. \label{errord-1}
		\end{equation}
		Combining \eqref{errora},  \eqref{errorb-1}, \eqref{errorc}, \eqref{errord-1}, we obtain
		\begin{align}
		&{\bf E}\Big\| \widehat H_{\beta_{\bf n}} -H  \Big\|^2_{L^2(\Omega)}\nn\\
		& \le {\bf E}A_{1,1}+ A_{1,2}+A_{1,3} \nn\\
		&\le 8\pi^d  V_{max}^2  \frac{2 \pi^{d/2}}{ d 2^d \Gamma(d/2)}   \frac{  \beta_{\bf n}^{d/2} }{ \prod_{k=1}^d n_k }+\frac{ 16 \mathcal{C}^2 (\mu_1,...\mu_d) \pi^{d/2}}{ d  \Gamma(d/2)}  \Big\| H \Big\|^2_{\mathcal{H}^{\mu_0}(2\Omega) } \beta_{\bf n}^{d/2} \prod_{k=1}^d  ( n_k)^{-4\mu_k}\nn\\
		&+4 \beta_{\bf n}^{-\mu_0} \Big\| H \Big\|^2_{\mathcal{H}^{\mu_0}(\Omega) }\nn\\
		&\le \overline C (\mu_1,...\mu_d, H) \beta_{\bf n}^{d/2} \prod_{k=1}^d  ( n_k)^{-4\mu_k}+ 4\beta_{\bf n}^{-\mu_0} \Big\| H \Big\|^2_{\mathcal{H}^{\mu_0}(\Omega) }.\nn
		\end{align}
		By a similar method used in the first part of the proof we can apply the previous results using Lemma \ref{lem2}, we immediately obtain the second error estimation.
		\end{proof}

		\section{Backward problem for parabolic equation with constant  coefficients}
		\setcounter{equation}{0}
		In this section, we consider  the problem of recovering  ${\bf u}({\bf x},t), 0\le t <T,$  such that
		\bq
		\left\{ \begin{gathered}
			{\bf u}_t({\bf x},t )+\A{\bf u}({\bf x}, t)  = F({\bf u}({\bf x},t))+G({\bf x},t) ,~~0<t<T, \ \ {\bf x} \in \Omega, \hfill \\
			{\bf u}({\bf x},t)= 0,~~{\bf x} \in \partial \Omega,\hfill\\
			{\bf u}({\bf x},T)= H({\bf x}), \hfill\\
		\end{gathered}  \right. \label{problem01111}
		\eq
		where  $H  \in L^2(\Omega)$ is a given function. The operator $\A$  solves the following eigenvalue problem
		\begin{align*}
			\A \psi_{\bf p}({\bf x})=  {\bf\ M} (|{\bf p}|) \psi_{\bf p}({\bf x}),
		\end{align*}
		for a non-decreasing function ${\bf M} $ and the eigenfunctions $ \psi_{\bf p}({\bf x})$ defined in \eqref{eigenfunctions-laplacian}.\\
		\noindent Now, we give some examples of  operators  $\A$ defined by the spectral theorem  using the Laplacian in $\Omega$ and the corresponding  eigenvalues.
		\begin{example}
			\begin{enumerate}[{ \bf \upshape(a)}]	
				
				\item If $\A=-\Delta$ { in $\Omega$ with Dirichlet boundary conditions} then Problem \eqref{problem01111} is called nonlinear heat equation and    its eigenvalues  are $$ {\bf M} (|{\bf p}|)=|{\bf p}|^2=p_1^2+...+p_d^2. $$
				In this case the eigenfunctions are given by equation \eqref{eigenfunctions-laplacian}.

				\item If $\A=\Delta^2 $ in $\Omega$ with Dirichlet boundary conditions
				then  Problem \eqref{problem01111} is called nonlinear biharmonic heat equation (see \cite{Mo}) and using the spectral theorem
				its eigenvalues  are $$ {\bf M} (|{\bf p}|)=|{\bf p}|^4=\Big(p_1^2+...+p_d^2\Big)^2. $$
				In this case  again, the eigenfunctions are given by equation \eqref{eigenfunctions-laplacian}.
				
				\item If $\A=-\Delta+\Delta^2$ in $\Omega$ with Dirichlet boundary conditions then Problem \eqref{problem01111} is called  extended Fisher-Kolmogorov equation (see \cite{He,Kwa}) and using the spectral theorem   its eigenvalues  are
				\[ {\bf M} (|{\bf p}|)=|{\bf p}|^2+|{\bf p}|^4= p_1^2+...+p_d^2+ \Big( p_1^2+...+p_d^2 \Big)^2.
				\]
				In this case  again, the eigenfunctions are given by equation \eqref{eigenfunctions-laplacian}.
				
				\item If $\A=2\Delta+\Delta^2$ in $\Omega$ with Dirichlet boundary conditions then Problem \eqref{problem011} is called Swift-Hohenberg equation (see \cite{Gio}) and using the spectral theorem   its eigenvalues  are
				\[ {\bf M} (|{\bf p}|)=-2|{\bf p}|^2+|{\bf p}|^4=-2\Big( p_1^2+...+p_d^2\Big)+ \Big( p_1^2+...+p_d^2 \Big)^2.
				\]
				In this case  again, the eigenfunctions are given by equation \eqref{eigenfunctions-laplacian}.
			\end{enumerate}
		\end{example}
		
		\begin{proposition}  \label{proposition3.1}
			If Problem \eqref{problem01111} has a unique solution ${\bf u}$ then
			it satisfies that
			\begin{align}
				{\bf u}({\bf x},t)&=  	\sum_{p \in  {\mathbb N}^d}  \Big[ e^{(T-t) {\bf M} (|{\bf p}|) } H_{\bf p}  -\int_t^T e^{(\tau-t) {\bf M} (|{\bf p}|) }  G_{\bf p}(\tau)d\tau\Big] \psi_{\bf p}({\bf x})\nn\\
				-&\sum_{{\bf p}\in  {\mathbb N}^d}  \Big[ \int_t^T e^{(\tau-t) {\bf M} (|{\bf p}|)  }  F_{\bf p}(	u)(\tau)d\tau\Big] \psi_{\bf p}({\bf x}). \label{resol3}
			\end{align}
			Where  $G_{\bf p}(\tau)=\big<G(\cdot, \tau),\psi_{\bf p}\big>, H_{\bf p}=\big<H,\psi_{\bf p}\big>, F_{\bf p}(	{\bf u})(\tau)=\big<F({\bf u}(\cdot, \tau)),\psi_{\bf p}\big>$ are the Fourier coefficients.
		\end{proposition}
		
		\begin{proof}
			{  The proof is similar to Theorem 3.1, page 2975, \cite{Tuan2011}  and we omit it here.}
		\end{proof}
		Let's define ${\mathbb P}_{\overline M}$  as the operator of  orthogonal projection onto the eigenspace spanned by the set $\Big\{  \psi_{\bf p}({\bf x}), 1 \le |{\bf p}|^2 \le \overline M \Big\}$ for a given $\overline M$ i.e,
		\begin{equation}  \label{projection}
			{\mathbb P}_{\overline M} v ({\bf x}):=  \sum_{|{\bf p}|^2 \le \overline M   } \Big< v, \psi_{\bf p} \Big>_{L^2(\Omega)}\psi_{\bf p}({\bf x}),~~\text{for all}~~v \in L^2(\Omega).
		\end{equation}
	Our method in this subsection is described as follows: First, we approximate  the two functions $H$ and $G$ by $H_{\beta_{\bf n}}$ and  $G_{\beta_{\bf n}}$ which are defined in Theorem \eqref{theorem2.1}. Then we use the Fourier truncation method by adding the operator ${\mathbb P}_{\rho_{\bf n}}$  and  introducing  the following regularized problem
		\bq
		\left\{ \begin{gathered}
			\frac{\partial \overline U_{\rho_{\bf n}, \beta_{\bf n} }({\bf x},t) }{\partial t}+\A \overline U_{\rho_{\bf n}, \beta_{\bf n}}({\bf x},t)  = 	{\mathbb P}_{\rho_{\bf n}} F( \overline U_{\rho_{\bf n}, \beta_{\bf n} } ({\bf x},t))+{\mathbb P}_{\rho_{\bf n}} G({\bf x},t) ,~~0<t<T, \hfill \\
			\overline U_{\rho_{\bf n}, \beta_{\bf n} }({\bf x},t)= 0,~~{\bf x} \in \partial \Omega,\hfill\\
			\overline U_{\rho_{\bf n}, \beta_{\bf n}} ({\bf x},T)= {\mathbb P}_{\rho_{\bf n}} \widehat H_{\beta_{\bf n}}({\bf x}), \hfill\\
		\end{gathered}  \right. \label{regularizedproblem}
		\eq
		where $\rho_{\bf n}$ is called the   regularization parameter. 	The function  $\rho:\mathbb{N}^d\to\mathbb{R}$  is a function that depend on $\beta_{\bf n}$ and we use the notation $ \rho_{\bf n}=\rho({\bf n})$. Noting that $\lim_{{\bf n} \to 0} \rho_{ {\bf n}}=\infty$.
			Define the set  $ 	\mathcal{W}_{  \rho_{\bf n}} $ for any ${\bf n}=(n_1,..n_d)\in \mathbb{N}^d$
			\begin{equation}
			\mathcal{W}_{  \rho_{\bf n}}=  \Big\{ {\bf p} =(p_1,...p_d) \in \mathbb{N}^d  : |{\bf p}|^2= \sum_{k=1}^d p_k^2 \le \rho_{\bf n}= \rho(n_1,...n_d)  \Big\}.
			\end{equation}
		Since the solution of  Problem \eqref{regularizedproblem} depends on two terms $\bn$ and  $\rho_{\bf n}$, we denote it by $\overline U_{\rho_{\bf n}, \beta_{\bf n}}$.
		
		\begin{theorem} \label{theorem3.1}
			Suppose that  $ \bn:= \beta(n_1,n_2,...n_d),~ \rho_{\bf n}:= \rho(n_1,n_2,...n_d)$ are such that
			\begin{equation}  \label{cond1}
			\lim_{ |{\bf n}| \to +\infty } \beta_{\bf n}=	\lim_{ |{\bf n}| \to +\infty } \rho_{\bf n}=+\infty , \quad 	\lim_{ |{\bf n}| \to +\infty } \frac{  e^{2T {\bf M}(\sqrt{\rho_{\bf n}})}  \beta_{\bf n}^{d/2} }{\prod_{k=1}^d  ( n_k)^{4\mu_k} }= 	\lim_{ |{\bf n}| \to +\infty }  e^{2T {\bf M}(\sqrt{\rho_{\bf n}})}  \beta_{\bf n}^{-\mu_0}=0,
			\end{equation}
			where we recall $|{\bf n}|=\sqrt{ \sum_{k=1}^d {n_k^2}}$. 	Assume that $H, G, H_{\bn}, G_{\bn}$ are  as in  Theorem \ref{theorem2.1}.
			  {  Suppose that  $F \in L^\infty (\mathbb{R})$ and  $F$ is a Lipschitz function, i.e. there exists a positive constant K such that
			  \begin{equation}\label{lipschitz-condition}
			  |F(\xi_1)- F(\xi_2)| \le K |\xi_1-\xi_2|,\quad \forall \xi_1, \xi_2 \in \mathbb{R}.
			  \end{equation}
			}
			The Problem \eqref{regularizedproblem} has a unique solution $\overline U_{\rho_{\bf n}, \beta_{\bf n}} \in C([0,T];L^2(\Omega))$ which  satisfies
			\begin{align}
				\overline U_{\rho_{\bf n}, \beta_{\bf n}}({\bf x},t)&=  	\sum_{{\bf p} \in  	\mathcal{W}_{  \rho_{\bf n}}  }  \Big[ e^{(T-t) {\bf M} (|{\bf p}|) }  \widehat H_{\beta_{\bf n},{\bf p}}  -\int_t^T  e^{(\tau-t) {\bf M} (|{\bf p}|)}  \widehat G_{\beta_{\bf n},{\bf p}}  (\tau)d\tau \Big] \psi_{\bf p}({\bf x})\nn\\
				&-\sum_{{\bf p} \in  	\mathcal{W}_{  \rho_{\bf n}} }  \Bigg[   \int_t^T  e^{(\tau-t) {\bf M} (|{\bf p}|)}   F_{\bf p}(	\overline U_{\rho_{\bf n}, \beta_{\bf n}})(\tau)d\tau \Bigg]\psi_{\bf p}({\bf x}). \label{resol}
			\end{align}
			{%\color{blue}
				where
				$\widehat G_{\beta_{\bf n},{\bf p}}(t) = \left \langle  \widehat G_{\beta_{\bf n}}(.,t), \psi_{\bf p} \right \rangle$ and  	$\widehat H_{\beta_{\bf n},{\bf p}} = \left \langle  \widehat H_{\beta_{\bf n}}, \psi_{\bf p} \right \rangle$.
			}

			\begin{enumerate}[{ \bf \upshape(a)}]	
			{
				\item Assume that the problem \eqref{problem01111} has a unique solution ${\bf u} $ such that
				\begin{equation}
					\sum_{{\bf p} \in {\mathbb N}^d } e^{2t {\bf M}(|{\bf p}|)  }    {\bf u}_{\bf p}^2(t): = \widetilde A_1<\infty,~\text{for all}~t \in [0,T].
				\end{equation}
				Then as $|{\bf n}|\to \infty$, $	{\bf E}	\Big\| \overline U_{\rho_{\bf n}, \beta_{\bf n}}(.,t)-{\bf u}(.,t) \Big\|_{L^2(\Omega)}^2 \quad \text{ is of order }$
					\begin{equation}
					e^{-2t {\bf M}(\sqrt{\rho_{\bf n}})}  	\max \Bigg( \frac{ e^{2T {\bf M}(\sqrt{\rho_{\bf n}})}   \beta_{\bf n}^{d/2} }{\prod_{k=1}^d  ( n_k)^{4\mu_k} }, e^{2T {\bf M}(\sqrt{\rho_{\bf n}})}   \beta_{\bf n}^{- \mu_0} , 1 \Bigg)  .
					\end{equation}
				\item  Assume that the problem \eqref{problem01111} has unique solution  ${\bf u} $ such that
				\begin{equation}
					\sum_{{\bf p} \in {\mathbb N}^d }  \big |{\bf M}(|{\bf p}|) \big|^{\alpha} e^{2t {\bf M}(|{\bf p}|)  }    {\bf u}_{\bf p}^2(t):=\widetilde A_2<\infty,
				\end{equation}
				for any $\alpha>0$ and $t \in [0,T]$.
				Then as $|{\bf n}|\to \infty$,
 ${\bf E}	\Big\| \overline U_{\rho_{\bf n}, \beta_{\bf n} }(.,t)-{\bf u}(.,t) \Big\|_{L^2(\Omega)}^2$ { is of order }
				\begin{equation}
				e^{-2t {\bf M}(\sqrt{\rho_{\bf n}})}  	\max \Bigg( \frac{ e^{2T {\bf M}(\sqrt{\rho_{\bf n}})}   \beta_{\bf n}^{d/2} }{\prod_{k=1}^d  ( n_k)^{4\mu_k} },  e^{2T {\bf M}(\sqrt{\rho_{\bf n}})}   \beta_{\bf n}^{- \mu_0}, \Big|{\bf M}(\sqrt{\rho_{\bf n}})  \Big|^{-2\al}  \Bigg).
				\end{equation}
				for all $t \in [0,T]$.
				\item  Assume that the problem \eqref{problem01111} has a unique solution  ${\bf u} $ such that
				\begin{equation}
					\sum_{{\bf p} \in {\mathbb N}^d }   e^{2(t+\delta) {\bf M }(|{\bf p}|)   }    {\bf u}_{\bf p}^2 =\widetilde A_3<\infty,
				\end{equation}
				for any real number $\delta \ge 0$ and $t \in [0,T].$
				Then as $|{\bf n}|\to \infty$,
 $	{\bf E}	\Big\| \overline U_{\rho_{\bf n}, \beta_{\bf n}}(.,t)-{\bf u}(.,t)  \Big\|_{L^2(\Omega)}^2$   is of order
				\begin{equation}
				e^{-2t {\bf M}(\sqrt{\rho_{\bf n}})}  	\max \Bigg( \frac{ e^{2T {\bf M}(\sqrt{\rho_{\bf n}})}   \beta_{\bf n}^{d/2} }{\prod_{k=1}^d  ( n_k)^{4\mu_k} },  e^{2T {\bf M}(\sqrt{\rho_{\bf n}})}   \beta_{\bf n}^{- \mu_0},  e^{-2\delta     {\bf M}(\sqrt{\rho_{\bf n}}) } \Bigg).
				\end{equation}
			 	
			}

			\end{enumerate}
		\end{theorem}
		
		\begin{proof}[{\bf Proof of Theorem \ref{theorem3.1}}]
			We divide the proof into some  smaller parts.\\
			{\bf Part 1.} {\it 	The nonlinear integral equation \eqref{resol} has unique solution $	\overline U_{\rho_{\bf n}} \in C([0,T];L^2(\Omega))$.}\\
			The proof is similar to \cite{Tuan}( { See Theorem 3.1, page 2975 \cite{Tuan2011}}). Hence, we omit it here. \\
			{\bf Part 2.} {\it The error estimate
				$
				{\bf E}	\Big\| \overline U_{\rho_{\bf n}, \beta_{\bf n}}-{\bf u} \Big\|_{L^2}.
				$}

			First, using Parseval's identity, equations \eqref{resol3}, and \eqref{resol} we get
			\begin{align}
				\Big\| \overline U_{\rho_{\bf n}, \beta_{\bf n}}-{\bf u} \Big\|_{L^2(\Omega)}^2& \le  4 	\sum_{{\bf p} \in   \mathcal{W}_{\rho_{\bf n}} } \Bigg[ e^{(T-t) {\bf M} (|{\bf p}|)}  \left(\widehat H_{\beta_{\bf n},{\bf p}}-H_{\bf p} \right)   \Bigg]^2  \nn\\
				& +  4 	\sum_{{\bf p} \in  	\mathcal{W}_{  \rho_{\bf n}} } \Bigg[ \int_t^T e^{(\tau-t) {\bf M} (|{\bf p}|) }  \left(\widehat G_{\beta_{\bf n},{\bf p}}(\tau)-G_{\bf p}(\tau)d\tau \right)   \Bigg]^2  \nn\\
				&+4 \sum_{{\bf p} \in  	\mathcal{W}_{  \rho_{\bf n}} }  \Bigg[  \int_t^T e^{(\tau-t) {\bf M} (|{\bf p}|)}    \left(F_{\bf p}(\overline U_{\rho_{\bf n}, \beta_{\bf n}}	)- F_{\bf p}({\bf u})(s) \right)ds \Bigg]^2\nn\\
				&+ 4\sum_{{\bf p} \notin  	\mathcal{W}_{  \rho_{\bf n}} } \big| {\bf u}_{\bf p} \big|^2. \label{er1}
			\end{align}
			Using the Cauchy-Schwartz inequality, the expectation of the right hand side of \eqref{er1} is bounded by
			\begin{align}
				&{\bf E}	\Big\| \overline U_{\rho_{\bf n}, \beta_{\bf n}}-{\bf u} \Big\|_{L^2(\Omega)}^2\nn\\
				& \le  4  e^{2(T-t) {\bf M }(\sqrt{\rho_{\bf n}})} {\bf E}\left(	\sum_{{\bf p} \in  	\mathcal{W}_{  \rho_{\bf n}} }
				\Big[ \left(\widehat H_{\beta_{\bf n},{\bf p}}-H_{\bf p} \right)   \Big]^2   \right)\nn\\
				& +  4  \int_t^T e^{2(\tau-t){\bf M}(\sqrt{\rho_{\bf n}})  }d\tau  {\bf E}\left(	\sum_{{\bf p} \in  	\mathcal{W}_{  \rho_{\bf n}} } \Big[ \int_t^T  \left(\widehat G_{\beta_{\bf n},{\bf p}}(\tau)-G_{\bf p}(\tau)\right)^2d\tau    \Big] \right) \nn\\
				&+4(T-t) \int_t^T e^{2(\tau-t) {\bf M}(\sqrt{\rho_{\bf n}})   }   \sum_{{\bf p} \in  	\mathcal{W}_{  \rho_{\bf n}} }  {\bf E}\left(\Big[   \left(F_{{\bf p}}(\overline U_{\rho_{\bf n}, \beta_{\bf n}}	)(\tau)- F_{\bf p}({\bf u})(\tau) \right) \Big]^2 \right)ds\nn\\
				&+4\sum_{{\bf p} \notin  	\mathcal{W}_{  \rho_{\bf n}} } \big| {\bf u}_{\bf p} \big|^2.
			\end{align}
			It  follows from the Lipschitz property of $F$ that
			\begin{align}
				{\bf E}	\Big\| \overline U_{\rho_{\bf n}, \beta_{\bf n}}-{\bf u} \Big\|_{L^2(\Omega)}^2 &\le  4 e^{2(T-t) {\bf M}(\sqrt{\rho_{\bf n}})}  {\bf E}	\Big\| \widehat H_{\beta_{\bf n}}-H \Big\|_{L^2(\Omega)}^2\nn\\
				&+4 T  e^{2(T-t) {\bf M}(\sqrt{\rho_{\bf n}})}  {\bf E}	\Big\| \widehat G_{\beta_{\bf n}}-G \Big\|_{  L^\infty( 0,T; L^2(\Omega))  }^2\nn\\
				&+4 T K^2 \int_t^T  e^{2(\tau-t) {\bf M}(\sqrt{\rho_{\bf n}})}   {\bf E}	\Big\| \overline U_{\rho_{\bf n}, \beta_{\bf n} }(.,\tau)-{\bf u} (.,\tau)\Big\|_{L^2(\Omega)}^2d\tau \nn\\
				& +4\sum_{{\bf p } \notin  	\mathcal{W}_{  \rho_{\bf n}} } \big| {\bf u}_{\bf p} \big|^2.\label{exp-ineq-difference}
			\end{align}
			Now, we deal with the  three cases.\\
			{\bf Case 1}. Assume that the series
			\begin{align}
				\sum_{{\bf p} \in {\mathbb N}^d } e^{2t {\bf M}(|{\bf p}|)  }  {\bf u}_{\bf p}^2
			\end{align}
			converges  to $\widetilde A_1$. Then
			multiplying both sides of the  inequality \eqref{exp-ineq-difference} by $e^{2t {\bf M}(\sqrt{\rho_n})} $, we obtain that
			\begin{eqnarray}
			\begin{aligned}
				&e^{2t {\bf M }(\sqrt{\rho_{\bf n} })}   {\bf E}	\Big\| \overline U_{\rho_{\bf n}, \beta_{\bf n} }(.,t)-{\bf u}(.,t) \Big\|_{L^2(\Omega)}^2\nn\\
				 &\le  4  e^{2T {\bf M }(\sqrt{\rho_{\bf n}})}  \left( {\bf E}	\Big\| \widehat H_{\beta_{\bf n}}-H \Big\|_{L^2(\Omega)}^2+T {\bf E}	\Big\| \widehat G_{\beta_{\bf n}}-G \Big\|_{L^\infty(0,T;L^2(\Omega))}^2\right)+4 \widetilde A_1\nn\\
				&+4 T K^2 \int_t^T  e^{2\tau {\bf M}(\sqrt{\rho_{\bf n}})}     {\bf E}	\Big\| \overline U_{\rho_{\bf n}, \beta_{\bf n}}(.,\tau)-{\bf u} (.,\tau)\Big\|_{L^2(\Omega)}^2d\tau,
			\end{aligned}
			\end{eqnarray}
			where we used the fact  that $$\sum_{{\bf p} \notin  	\mathcal{W}_{  \rho_{\bf n}}} \big| {\bf u}_{\bf p} \big|^2= \sum_{{\bf p} \notin  \mathcal{W}_{  \rho_{\bf n}} }   e^{-2t {\bf M}(|{\bf p}|)  } e^{2t {\bf M}(|{\bf p}|)  }  \big| {\bf u}_{\bf p} \big|^2 \le 	e^{-2t  {\bf M }(\sqrt{\rho_{\bf n}})}   \widetilde A_1.$$
			Above we used the monotone  increasing property of ${\bf M}$.
			{ %\color{blue}
			
			Using the above Gronwall's inequality, } we obtain that
			\begin{align}
				&	e^{2t {\bf M}(\sqrt{\rho_{\bf n}})}    {\bf E}	\Big\| \overline U_{\rho_{\bf n}, \beta_{\bf n}}-{\bf u} \Big\|_{L^2(\Omega)}^2 \nn\\
				&\le  4  	e^{2T {\bf M }(\sqrt{\rho_{\bf n}})}  \left( {\bf E}	\Big\| \widehat H_{\beta_{\bf n}}-H \Big\|_{L^2(\Omega)}^2+T {\bf E}	\Big\| \widehat G_{\beta_{\bf n}}-G \Big\|_{L^\infty(0,T;L^2(\Omega))}^2\right) e^{4TK^2(T-t)}\nn\\
				&+ 4 \widetilde A_1 e^{4TK^2(T-t)}.
			\end{align}
			This implies that
			\begin{align}
				&  {\bf E}	\Big\| \overline U_{\rho_{\bf n}, \beta_{\bf n}}-{\bf u} \Big\|_{L^2(\Omega)}^2 \nn\\
				&\le \underbrace{ 4 e^{4TK^2(T-t)}}_{C5^2} 	e^{2(T-t) {\bf M}(\sqrt{\rho_{\bf n}})}  \left( {\bf E}	\Big\| \widehat H_{\beta_{\bf n}}-H \Big\|_{L^2(\Omega)}^2+T {\bf E}	\Big\| \widehat G_{\beta_{\bf n}}-G \Big\|_{L^\infty(0,T;L^2(\Omega))}^2\right) \nn\\
				&+ 4 \widetilde A_1 e^{4TK^2(T-t)} 	e^{-2t {\bf M}(\sqrt{\rho_{\bf n}})} .
			\end{align}
	{
			  It follows from Corollary \eqref{corollary2.1} that $	{\bf E}	\Big\| \overline U_{\rho_{\bf n}, \beta_{\bf n}}(.,t)-{\bf u}(.,t) \Big\|_{L^2(\Omega)}^2 $ \text{ is of order }
			\begin{equation}
			 	e^{-2t {\bf M}(\sqrt{\rho_{\bf n}})}  	\max \Bigg( \frac{ e^{2T {\bf M}(\sqrt{\rho_{\bf n}})}   \beta_{\bf n}^{d/2} }{\prod_{k=1}^d  ( n_k)^{4\mu_k} }, e^{2T {\bf M}(\sqrt{\rho_{\bf n}})}   \beta_{\bf n}^{- \mu_0} , 1 \Bigg)  .
			\end{equation}
				}
			
			{\bf Case 2}. Suppose  that the series
			\begin{align}
				\sum_{{\bf p} \in {\mathbb N}^d }  \Big|{ {\bf M }(|\bf p|)}\Big|^{2\alpha} e^{2t {\bf M }(|\bf p|) }  {\bf u}_{\bf p}^2
			\end{align}
			converges to $\widetilde A_2$.	By a similar technique as in case 1 above  and using the following estimate
			\begin{align}
				\sum_{{\bf p} \notin  \mathcal{W}_{  \rho_{\bf n} }  } \big| {\bf u}_{\bf p} \big|^2&= \sum_{{\bf p} \notin  	\mathcal{W}_{  \rho_{\bf n}} }  \Big|{ {\bf M }(|\bf p|)}\Big|^{-2\alpha}     e^{-2t {\bf M }(|\bf p|) }    \Big|{ {\bf M}(|\bf p|)}\Big|^{2\alpha}    e^{2t {\bf M }(|\bf p|) }
				\big| {\bf u}_{\bf p} \big|^2 \nn\\
				&\le \Big|{\bf M}(\sqrt{\rho_{\bf n}})  \Big|^{-2\al}	e^{-2t {\bf M}(\sqrt{\rho_{\bf n}})  }  \widetilde A_2,
			\end{align}
			we deduce  that
			\begin{align}
				&  {\bf E}	\Big\| \overline U_{\rho_n, \beta_{\bf n} }-{\bf u} \Big\|_{L^2(\Omega)}^2 \nn\\
				&\le \underbrace{ 4 e^{4TK^2(T-t)}}_{C5^2}  e^{2(T-t) {\bf M}(\sqrt{\rho_{\bf n}})  }   \left( {\bf E}	\Big\| \widehat H_{\beta_{\bf n}}-H \Big\|_{L^2(\Omega)}^2+T {\bf E}	\Big\| \widehat G_{\beta_{\bf n}}-G \Big\|_{L^\infty(0,T;L^2(\Omega))}^2\right)\nn\\
				& + 4 \widetilde A_2 e^{4TK^2(T-t)}   \Big|{\bf M}(\sqrt{\rho_{\bf n}})  \Big|^{-2\al}	e^{-2t {\bf M}(\sqrt{\rho_{\bf n}})  }  .
			\end{align}
	{  It follows from Corollary \eqref{corollary2.1} that $	{\bf E}	\Big\| \overline U_{\rho_{\bf n}, \beta_{\bf n}}(.,t)-{\bf u}(.,t) \Big\|_{L^2(\Omega)}^2 $ \text{ is of order }
		\begin{equation}
		 	e^{-2t {\bf M}(\sqrt{\rho_{\bf n}})}  	\max \Bigg( \frac{ e^{2T {\bf M}(\sqrt{\rho_{\bf n}})}   \beta_{\bf n}^{d/2} }{\prod_{k=1}^d  ( n_k)^{4\mu_k} },  e^{2T {\bf M}(\sqrt{\rho_{\bf n}})}   \beta_{\bf n}^{- \mu_0}, \Big|{\bf M}(\sqrt{\rho_{\bf n}})  \Big|^{-2\al}  \Bigg).
		\end{equation}
	}
		
			\noindent {\bf Case 3}. Suppose  that the series
			$$
			\sum_{{\bf p} \in {\mathbb N}^d }   e^{2(t+\delta) {\bf M}(|{\bf p}|)   }    {\bf u}_p^2
			$$
			converges  to $\widetilde A_3$.	By a similar technique as in case 1 above  and using the following estimate
		\begin{align}
			\sum_{{\bf p} \notin  	\mathcal{W}_{  \rho_{\bf n}} } \big| {\bf u}_{\bf p} \big|^2&= \sum_{{\bf p} \notin  	\mathcal{W}_{  \beta_{\bf n}} }   e^{-2(t+\delta ) {\bf M}(|{\bf p}|)   }      e^{2(t+\delta) {\bf M}(|{\bf p}|)   }    {\bf u}_{\bf p}^2  \nn\\
			& \le  	e^{-2(t+\delta)     {\bf M}(\sqrt{\beta_{\bf n}})     } \widetilde A_3,
			\end{align}
			we deduce that
			\begin{eqnarray}
			\begin{aligned}
				&  {\bf E}	\Big\| \overline U_{\rho_{\bf n}, \beta_{\bf n}}-{\bf u} \Big\|_{L^2(\Omega)}^2 \nn\\
				&\le \underbrace{ 4 e^{4TK^2(T-t)}}_{C5^2}  e^{2(T-t) {\bf M}(\sqrt{\rho_{\bf n}})  }   \left( {\bf E}	\Big\| \widehat H_{\beta_{\bf n}}-H \Big\|_{L^2(\Omega)}^2+{\bf E}	\Big\| \widehat G_{\beta_{\bf n}}-G \Big\|_{L^\infty(0,T;L^2(\Omega))}^2\right)\nn\\
				& \quad \quad \quad + 4 \widetilde A_3 e^{4TK^2(T-t)}   e^{-2(t+\delta)     {\bf M}(\sqrt{\rho_{\bf n}})     } .	
			\end{aligned}
			\end{eqnarray}		
				{  It follows from Corollary \eqref{corollary2.1} that $   {\bf E}	\Big\| \overline U_{\rho_n, \beta_{\bf n}}-{\bf u} \Big\|_{L^2(\Omega)}^2 $ is of order
					\begin{equation}
					e^{-2t {\bf M}(\sqrt{\rho_{\bf n}})}  	\max \Bigg( \frac{ e^{2T {\bf M}(\sqrt{\rho_{\bf n}})}   \beta_{\bf n}^{d/2} }{\prod_{k=1}^d  ( n_k)^{4\mu_k} },  e^{2T {\bf M}(\sqrt{\rho_{\bf n}})}   \beta_{\bf n}^{- \mu_0},  e^{-2\delta     {\bf M}(\sqrt{\rho_{\bf n}}) } \Bigg).
					\end{equation}
				}
		\end{proof}

		\begin{remark} \label{remma3.1}
				We give one choice for  $ \beta_{\bf n}$  and $ \rho_{\bf n}$ which satifies \eqref{cond1}. Let  $0< 2\al_0 < \mu_0 $ and $e^{2T {\bf M}(\sqrt{\rho_{\bf n}})}  = \beta_{\bf n}^{2 \al_0}$.  Since  $ \frac{ e^{2T {\bf M}(\sqrt{\rho_{\bf n}})}   \beta_{\bf n}^{d/2} }{\prod_{k=1}^d  ( n_k)^{4\mu_k} } \to 0$ when $|{\bf n}| \to +\infty$, we  can choose $\beta_{\bf n}$ such that
				\begin{equation}
				\lim_{|{\bf n}| \to +\infty}  \frac{ \beta_{\bf n}^{2\al_0 +d/2} }{ \prod_{k=1}^d  ( n_k)^{4\mu_k}    }=0.
				\end{equation}
				Let us choose 	$\beta_{\bf n}^{2\al_0 +d/2}  =\prod_{k=1}^d  n_k $ then $\beta_{\bf n} = \left( \prod_{k=1}^d  n_k  \right)^{\frac{1}{ 2 \al_0+d/2 }}$
				and then we choose $ \rho_{\bf n}$ such that
				\begin{equation}
				{\bf M}  (  \sqrt{\rho_{\bf n}} )= \frac{\al_0 }{T} \log \left(  \beta_{\bf n}\right)= \frac{\al_0 }{T (2 \al_0+d/2)} \log \left(  \prod_{k=1}^d  n_k \right).
				\end{equation}
				In above theorem, for  the case (b),  $	{\bf E}	\Big\| \overline U_{\rho_{\bf n}, \beta_{\bf n}}(.,t)-{\bf u}(.,t) \Big\|_{L^2(\Omega)}^2 $ \text{ is of order }
				\begin{equation}
				\left( \prod_{k=1}^d  n_k  \right)^{\frac{-4\al_0 t}{ 4 T\al_0+dT }}  	\max \Bigg( \frac{1 }{\prod_{k=1}^d  ( n_k)^{4\mu_k-1} }, \left( \prod_{k=1}^d  n_k  \right)^{\frac{2\al_0 - \mu_0}{ 2 \al_0+d/2 }} ,      \log^{-2\al} \left(  \prod_{k=1}^d  n_k \right) \Bigg).
				\end{equation}
				In above theorem, for the  case (c),   $	{\bf E}	\Big\| \overline U_{\rho_{\bf n}, \beta_{\bf n}}(.,t)-{\bf u}(.,t) \Big\|_{L^2(\Omega)}^2 $ \text{ is of order }
				\begin{equation}
				\left( \prod_{k=1}^d  n_k  \right)^{\frac{-4\al_0 t}{ 4 T\al_0+dT }}  	\max \Bigg( \frac{1 }{\prod_{k=1}^d  ( n_k)^{4\mu_k-1} }, \left( \prod_{k=1}^d  n_k  \right)^{\frac{2\al_0 - \mu_0}{ 2 \al_0+d/2 }} ,
				\left( \prod_{k=1}^d  n_k  \right)^{\frac{-4\al_0 \delta}{ 4 T\al_0+dT }}\Bigg).
				\end{equation}
				
			\end{remark}

	    \section{The backward problem  for parabolic equation with time dependent coefficients}\label{section4}
					\setcounter{equation}{0}
					\subsection{The problem with  coefficients that  depend only  on $t$}

					\noindent  In this section, we consider  the problem of constructing a solution  ${\bf u} \in C([0,T;\mathcal{V}(\Omega)]), {\bf u}'\in L^2(0,T;L^2(\Omega))$  such that
					${\bf u}$ satisfies the following parabolic equation with time dependent coefficients
					\bq
					\left\{ \begin{gathered}
						{\bf u}_t+\A(t){\bf u}  = F({\bf u}({\bf x},t))+G({\bf x},t) ,~~{\bf x} \in  \Omega, 0<t<T, \hfill \\
						{\bf u}({\bf x},t)= 0,~~{\bf x} \in \partial \Omega,\hfill\\
						{\bf u}({\bf x},T)= H({\bf x}), \quad {\bf x} \in  \Omega \hfill\\
					\end{gathered}  \right. \label{problem011}
					\eq
					where $H  \in L^2(\Omega)$.
					Here   $\mathcal V (\Omega)  \overset d \hookrightarrow L^2(\Omega)$;     i.e.,   ${\mathcal V}(\Omega)\subset  L^2(\Omega) $   is   continuously embedded    into $L^2(\Omega)$.  It  means that  there exists some constant $ m_0 >0$ such that for all $v\in  {\mathcal V}(\Omega)$
					\begin{equation}
						\|v\|_{L^2(\Omega)} \le m_0 	\|v\|_{{\mathcal V}(\Omega)}.
					\end{equation}
					Then $ L^2(\Omega)  \overset d \hookrightarrow  {\mathcal V}'(\Omega) $ via  $v\mapsto \big<v \, \vert \, \cdot \big>_{L^2(\Omega)}$, where  ${\mathcal V}'(\Omega)$  denotes the { dual   space of ${\mathcal V}(\Omega) $.}
					In this section, we assume that the source function
					$F: \mathbb{R} \to
					\mathbb{R}$ is a locally Lipschitz function i.e,  for each $Q>0$ and
					for any $u$, $v$ satisfying $|u|, |v| \leqslant Q$, there holds
					\begin{equation}\label{dkff}
						\left| {F(u) - F(v)} \right|
						\leqslant K_F(Q)\left| {u - v} \right|,
					\end{equation}
					where
					\begin{equation} \label{dkff1}
						K(Q): =
						\sup \left\{ {\left| {\frac{{F(u) - F(v)}}{{u - v}}}
							\right|:\left| u \right|,\left| v \right|
							\leqslant Q,u \ne v}
						\right\} <  + \infty.
					\end{equation}
					We note that the function $Q\to K(Q)$ is increasing and $\mathop {\lim }\limits_{Q \to
						+ \infty } K(Q) =  + \infty $.
					
					First, we give the following definition
					
					\begin{definition} \label{asumption1111}
						The pair of operators $( \A(t), {\bf P} ) $ satisfies Assumption $(A)$ if  	the following conditions hold
						\begin{enumerate}[{ \upshape(a)}]
							
							\item  For any $v\in L^2(\Omega)$, there exists an increasing function ${\overline  M} : \mathbb{R} \to \mathbb{R}^+$ such as	
							\begin{equation}
								{\bf P} v({\bf x},t)= \sum_{{\bf p} \in {\mathbb N}^d}  {\overline  M} (|{\bf p}|)\big\langle v, \psi_{\bf p}\big\rangle _ {L^2(\Omega)}\psi_{\bf p}({\bf x}),\quad v \in L^2(\Omega).	\label{asumption222}
							\end{equation}
							
							\item For $ u \in {\mathcal V}(\Omega), v \in {\mathcal V}(\Omega) $
							\begin{equation}
								\Big\langle ({\bf P}-\A(t))u,v \Big\rangle_{L^2(\Omega)} \le \widetilde  M_a \|u\|_{{\mathcal V}(\Omega)} \|v\|_{{\mathcal V}(\Omega) }, \label{asumption000}
							\end{equation}
							for some  constant $\widetilde M_a>0$.
							\item For any $v\in {\mathcal V}(\Omega) $, there exists $\widehat M>0$ such that
							\begin{equation}
								\Big\langle ({\bf P}-\A(t))v,v \Big\rangle_{L^2(\Omega)}\geq \widehat M \Vert v\Vert_{{\mathcal V}(\Omega)}^2. \label{asumption111}
							\end{equation}

						\end{enumerate}

					\end{definition}

					Now we state the following lemma  concerning an  estimate of ${\bf P}_{\rho_{\bf n}}$.
					\begin{lemma}  \label{lemma4.1}
						Let ${\bf P}_{\rho_{\bf n}}  $ be defined as follows
						\begin{equation}\label{op:p-rho}
							{\bf P}_{\rho_{\bf n}}  v({\bf x})=\sum_{ {\bf p} \in  \mathcal{W}_{  \rho_{\bf n}}  }  {\overline  M } (|{\bf p}|)\big\langle v, \psi_{\bf p}\big\rangle _{L^2(\Omega)}\psi_{\bf p}({\bf x}),\quad \text {for all} \quad v\in L^2(\Omega).
						\end{equation}
						The operator  ${\bf P}_{\rho_{\bf n} }$ is a linear, bounded operator, and satisfies that
						\begin{equation}
							\|{\bf P}_{\rho_{\bf n}}\|_{\mathbb{L} (L^2(\Omega);L^2(\Omega)) } \le   {\overline  M } (\sqrt{\rho_{\bf n}}),
						\end{equation}
						{ where  ${\mathbb{L} (L^2(\Omega);L^2(\Omega)) }$ is the space of all bounded linear operators from $L^2(\Omega)$ to $L^2(\Omega)$}.
					\end{lemma}

					\begin{proof}	
						For $v \in L^2(\Omega)$, since $\overline M$ is a non-decreasing function, we have
						\begin{align}
							\|{\bf P}_{\rho_{\bf n}}v\|_{L^2(\Omega)}^2& \quad =\quad \quad \sum_{{\bf p} \in   \mathcal{W}_{  \rho_{\bf n}} }  \Big| {\overline  M }(|{\bf p}|) \Big|^2\big|\big\langle v, \psi_{\bf p}\big\rangle _{L^2(\Omega)}^2 \nn\\
							&\quad  \le \quad  \Big|{\overline  M } (\sqrt{\rho_{\bf n}}) \Big|^2 \sum_{{\bf p} \in \mathcal{W}_{  \rho_{\bf n}}}  \big\langle v, \psi_{\bf p}\big\rangle _{L^2(\Omega)}^2   \le\quad   \Big|{\overline  M } (\sqrt{\rho_{\bf n}}) \Big|^2 \|v\|_{L^2(\Omega)}^2. \label{inequality1}
						\end{align}
					\end{proof}
					
					\begin{definition} \label{def4.2}
						Define a subspace of $L^2(\Omega)$ as follows
						\begin{equation}
							\mathcal{G}_\gamma(\Omega):= \Bigg\{ v\in L^2(\Omega):  \sum_{ {\bf p}  \in {\mathbb N}^d}  \Big| {\overline  M }(|{\bf p}|) \Big|^{2\gamma} e^{2T {\overline  M} (|{\bf p}|)} \big\langle v, \psi_{\bf p}\big\rangle _{L^2(\Omega)}^2  <\infty  \Bigg\}
						\end{equation}
						for any $\gamma \ge 0$. The norm of  $v \in \mathcal{G}_\gamma(\Omega)$  is given by
						\begin{equation}
							\|v\|_{\mathcal{G}_\gamma(\Omega)}= \sqrt{\sum_{{\bf p}  \in {\mathbb N}^d}  \Big| {\overline  M}(|{\bf p}|) \Big|^{2\gamma} e^{2T {\overline  M} (|{\bf p}|)} \big\langle v, \psi_{\bf p} \big\rangle _{L^2(\Omega)}^2 }.
						\end{equation}
					\end{definition}
					
					\begin{lemma}\label{lemma4.2}
						For $v \in \mathcal{G}_{1+\gamma}(\Omega),~~\gamma\ge 0$
						\begin{align}
							\|{\bf P}_{\rho_{\bf n}}v-{\bf P}v\| \le \Big| {\overline  M}(|{\sqrt{\rho_{\bf n}}}|)\Big|^{-\gamma}  e^{-T{\overline  M} (\sqrt{\rho_{\bf n}}) } \big\| v \big\|_{\mathcal{G}_{1+\gamma}(\Omega) }.
						\end{align}
					\end{lemma}
					\begin{proof}
						We have
						\begin{align*}
							\|{\bf P}_{\rho_{\bf n} }v-{\bf P}v\|_{L^2(\Omega)}^2&=\sum_{{\bf p} \notin   \mathcal{W}_{  \rho_{\bf n}} }  \Big| {\overline  M }(|{\bf p}|) \Big|^2\big|\big\langle v, \psi_{\bf p}\big\rangle _{L^2(\Omega)}^2 \nn\\
							&= ~\sum_{{\bf p} \notin  \mathcal{W}_{  \rho_{\bf n}}   } \Big| {\overline  M}(|{\bf p}|)\Big|^{-2\gamma} e^{-2T {\overline  M } (|{\bf p}|)}    \Big| {\overline  M }(|{\bf p}|) \Big|^{2+2\gamma}  e^{2T {\overline  M } (|{\bf p}|)} \big\langle v, \psi_{\bf p}\big\rangle _{L^2(\Omega)}^2 \nn\\
							&\le ~\Big| {\overline  M }(|{\sqrt{\rho_{\bf n}}}|)\Big|^{-2\gamma}  e^{-2T{\overline M } (\sqrt{\rho_{\bf n}}) }  \big\| v \big\|^2_{ \mathcal{G}_{1+\gamma}(\Omega) }.
						\end{align*}
					\end{proof}
					
					\subsubsection{\bf The regularized solution and convergence rates}	
					Since ${\bf P}$ is an unbounded operator on $L^2(\Omega)$, we
					approximate it by  the following operator ${\bf P}_{\rho_{\bf n}}$ defined above in equation \eqref{op:p-rho}.
					Since  $\A(t)$ is an unbounded operator, we approximate it by a new approximate operator $ \A(t)-{\bf P}+{\bf P}_{\rho_{\bf n}}  $. Moreover since $F$ is a locally Lipschitz source function,  we approximate $F$ by $ F_{Q}$ defined by
					\begin{align} \label{locally}
						F_Q\left(w({\bf x},t) \right)
						=
						\begin{cases}
							F(Q), &\quad w({\bf x},t) >Q,\\
							F(w({\bf x},t)), &\quad - Q \le { w}({\bf x},t)\le Q, \\
							F(-Q),&\quad w({\bf x},t) < - Q.
						\end{cases}
					\end{align}
					for any $Q>0$.
					In the sequel we use  a parameter $Q_{\bf n}:=Q(n_1,n_2,...n_d) \to +\infty $ as $|{\bf n}| \to +\infty$. So, when $ {\bf n}$ large enough, we have that  $Q_{\bf n} \ge \|{\bf u}\|_{L^\infty (0,T; L^2(\Omega))} $.  Moreover, we also have
					\begin{equation}
						F_{Q_{\bf n} }( {\bf u} (x,t) )= F( {\bf u} (x,t) ), ~~\text{for}~ |{\bf n}| ~~\text{large enough}.
					\end{equation}
			{Using observation on p. 1250 in  \cite{Tuan3},		we also obtain  that 	$F_{Q_{\bf n}}$ is a globally Lipschitz source  function in the following sense
					\begin{equation}  \label{locally1}
						\| F_{Q_{\bf n}} (v_1)- F_{Q_{\bf n}}(v_2) \|_{L^2(\Omega)} \le 2K(Q_{\bf n}) \|v_1-v_2\|_{L^2(\Omega)},~v_1, v_2 \in L^2(\Omega).
					\end{equation}	
					We consider a  regularized problem  below	
					\bq
					\left\{ \begin{gathered}
						\frac{\partial \overline U_{\rho_{\bf n}, \beta_{\bf n}   }}{\partial t}+\A(t)\overline U_{\rho_{\bf n}, \beta_{\bf n}  } -{\bf P}\overline U_{\rho_{\bf n}, \beta_{\bf n}  } +{\bf P}_{\rho_{\bf n} }\overline U_{\rho_{\bf n} , \beta_{\bf n} } \\
						\quad \quad   = F_{Q_{\bf n}}({\overline U_{\rho_{\bf n}, \beta_{\bf n}  } }({\bf x},t))+\widehat G_{\beta_{\bf n} }({\bf x},t) ,~~0<t<T, \hfill \\
						{\overline U_{\rho_{\bf n}, \beta_{\bf n}  } }( {\bf x},t)= 0,~~{\bf x} \in \partial \Omega,\hfill\\
						{\overline U_{\rho_{\bf n}, \beta_{\bf n}  } }( {\bf x},T)=\widehat H_{\beta_{\bf n}}({\bf x}). \hfill\\
					\end{gathered}  \right. \label{regu0111}
					\eq
					In the following Theorem,  we
					obtain the existence, uniqueness and continuous dependence of the solutions for the
					proposed problem. We state the error estimation between the regularized solution and the exact solution.
					Our main result in this section is as follows
					{
						\begin{theorem}  \label{theorem4.1}
							Let $H, G, H_{\bn}, G_{\bn}$ be as in Theorem \ref{theorem2.1}.   Let us   choose  $\bn, ~\rho_{\bf n}$ such that
							\begin{equation}
								\lim_{ |{\bf n}| \to +\infty}
								e^{2T{\overline  M} (\sqrt{\rho_{\bf n} })}    \beta_{\bf n}^{d/2} \prod_{k=1}^d  ( n_k)^{-4\mu_k} =\lim_{|{\bf n}| \to +\infty}   e^{2T{\overline  M} (\sqrt{\rho_{\bf n} })  } \beta_{\bf n}^{-\mu_0} =0.
							\end{equation}
							Then Problem \eqref{regu0111} has a unique solution ${\overline U_{\rho_{\bf n}, \beta_{\bf n}} } \in C([0,T];L^2(\Omega))\cap L^2(0,T;{\mathcal V}(\Omega))$.
							Assume that  Problem  \eqref{problem011} has unique solution $ {\bf u} \in C([0,T];L^2(\Omega))  \cap   L^\infty(0,T; 	\mathcal{G}_{1+\gamma}(\Omega)) $ for any $\gamma \ge 0$.
							Choose $Q_{\bf n}$ such that
							\begin{equation}
								\lim_{ |{\bf n}| \to +\infty}   e^{4 K(Q_{\bf n})T } \overline \Pi({\bf n})=0
							\end{equation}
							where
							\begin{equation}\label{def:pi-n}
								\overline \Pi({\bf n}) = \max \Bigg(  e^{2T{\overline  M} (\sqrt{\rho_{\bf n}})}    \beta_{\bf n}^{d/2} \prod_{k=1}^d  ( n_k)^{-4\mu_k}, e^{2T{\overline  M } (\sqrt{\rho_{\bf n}})  } \beta_{\bf n}^{-\mu_0} ,  \Big| {\overline  M}(|{\sqrt{\rho_{\bf n}}}|)\Big|^{-2\gamma} \Bigg).
							\end{equation}
							Then  as $|{\bf n}|\to\infty$  the error
							\begin{equation}  \label{error111}
								{\bf E}	\Big\| \overline U_{\rho_{\bf n},\bn }(.,t)-u(.,t) \Big\|_{L^2(\Omega)}^2 \quad \text{ is of order } \quad 		e^{\Big( 4K(Q_{\bf n} )+2  \Big) (T-t)  }  e^{-2t{\overline  M } (\sqrt{\rho_{\bf n} })} 	\overline \Pi({\bf n}) .
							\end{equation}
						\end{theorem}
					}
					
					\begin{remark}  \label{remark4.1}
						Thanks to Remark \eqref{remma3.1}, 	we give one choice for  $ \beta_{n}$ as follows
						\begin{equation} \label{betan}
							\beta_{\bf n} = \left( \prod_{k=1}^d  n_k  \right)^{\frac{1}{ 2 \al_0+d/2 }}
						\end{equation}
					{ 	where $0<\al_0 < \frac{\mu_0}{2}$. }
						Then we choose $ \rho_{\bf n}$ such that
						\begin{equation}
							{\overline  M}  (  \sqrt{\rho_{\bf n}} )= \frac{\al_0 }{T (2 \al_0+d/2)} \log \left(  \prod_{k=1}^d  n_k \right).
						\end{equation}
						A simple computation gives that
						\begin{equation} \label{Pin}
							\overline \Pi({\bf n})=	\max \Bigg( \frac{1 }{\prod_{k=1}^d  ( n_k)^{4\mu_k-1} }, \left( \prod_{k=1}^d  n_k  \right)^{\frac{2\al_0 - \mu_0}{ 2 \al_0+d/2 }} ,
							\left( \prod_{k=1}^d  n_k  \right)^{\frac{-4\al_0 \delta}{ 4 T\al_0+dT }}\Bigg).
						\end{equation}
						Since 		\begin{equation}
							\lim_{ |{\bf n}| \to +\infty}   e^{4 K(Q_{\bf n})T } \overline \Pi({\bf n})=0
						\end{equation}
						we can take  $K(Q_{\bf n})$ such that  $  e^{4 K(Q_{\bf n})T }= \left(\overline \Pi({\bf n}) \right)^{\delta_0-1}  $
						for any $0<  \delta_0 <1$.  So, we have
						\begin{equation}  \label{4.26}
							K(Q_{\bf n}):= \frac{ \delta_0-1 }{4T} \log \left( \overline \Pi({\bf n}) \right).
						\end{equation}
					\end{remark}
				{ 	Since $\al_0 < \frac{\mu_0}{2}$ and $4\mu_k>1,~k=\overline{1,d}$, using   \eqref{Pin}, we deduce  that $\lim_{ |{\bf n}| \to \infty } \overline \Pi({\bf n})=0 $. Hence,   we need to  choose ${\bf n}$ large  enough such that  $\overline \Pi({\bf n})<1$.   So, the equality \eqref{4.26} is suitable which leads to a chosen $Q_{\bf n}$. }

				\noindent 	We state the two  corollaries of the Theorem \ref{theorem4.1} next
					\begin{corollary}  \label{corollary4.1}
						Let us take two functions $\Gamma_0, \Gamma_1$ which are continuous functions on $[0,T]$. Assume that
						\begin{align}
							m_0= \min \Big(  \min_{0 \le t \le T} |\Gamma_0(t)|,~  \min_{0 \le t \le T} |\Gamma_1(t)|\Big) >0
						\end{align}
						and
						\begin{align}
							m_1= \max \Big(  \max_{0 \le t \le T} |\Gamma_0(t)|,~  \max_{0 \le t \le T} |\Gamma_1(t)|\Big) >0.
						\end{align}
						In Problem \eqref{problem011}, let  $\A(t) {\bf u}=  -\Gamma_0 (t) \Delta {\bf u}+ \Gamma_1 (t) \Delta^2 {\bf u}$ and $F({\bf u})=  {\bf u}-{\bf u}^3$.  Then we get the backward in time problem for  {\bf extended Fisher-Kolmogorov equation} with time dependent coefficients  as follows
						\bq
						\left\{ \begin{gathered}
							{\bf u}_t -\Gamma_0 (t) \Delta {\bf u}+ \Gamma_1 (t) \Delta^2 {\bf u} = {\bf u}-{\bf u}^3+G({\bf x},t) ,  {\bf x} \in \Omega, 0<t<T, \hfill \\
							{\bf u} ({\bf x, t})= 	\Delta {\bf u} ({\bf x, t})=0,~ {\bf x} \in \partial \Omega, 0 \le t\le T,\hfill\\
							{\bf u}({\bf x},T)= H({\bf x}), ~{\bf x} \in \Omega .\hfill\\
						\end{gathered}  \right. \label{Fisher-Kolmogorov}
						\eq
						From  Definition \eqref{asumption1111},   	let  the operator ${\bf P}$    be   as follows
						\begin{equation} \label{P1}
							{\bf P} v := m_1\sum_{{\bf p} \in {\mathbb N}^d  } \big({\bf p}^2+{\bf p}^4 \big) <v, \psi_{\bf p}> \psi_{\bf p}.
						\end{equation}
						for any $v \in L^2(\Omega)$. It is easy to show that 	the pair of operators $( \A(t), {\bf P} ) $  as above satisfies Assumption $(A)$ in Defintion \eqref{asumption1111}. It is easy to see that the eigenvalues of ${\bf P}$ are
						$\overline M ({\bf p})= m_1 \big({\bf p}^2+{\bf p}^4 \big) $.
						Next, we find the operator  ${\bf P}_{\rho_{\bf n}}$ by truncating Fourier series in
						\eqref{P1} and we have
						\begin{equation}
							{\bf P}_{\rho_n} v:=  m_1 \sum_{ {\bf |p|} \le   \sqrt{\frac{\rho_{\bf n}} {  m_1 }  } }  \Big({\bf p}^2+{\bf p}^4\Big) <v, \psi_{\bf p}> \psi_{\bf p}({\bf x}).
						\end{equation}
						Thanks to \eqref{regu0111}, a  regularized problem for \eqref{Fisher-Kolmogorov} is given below
						\bq
						\left\{ \begin{gathered}
							\frac{\partial \overline U_{\rho_{\bf n}, \beta_{\bf n}   }}{\partial t} -\Gamma_0 (t) \Delta \overline U_{\rho_{\bf n}, \beta_{\bf n} }  + \Gamma_1 (t) \Delta^2 \overline U_{\rho_{\bf n}, \beta_{\bf n} } -{\bf P}\overline U_{\rho_{\bf n}, \beta_{\bf n}  } +{\bf P}_{\rho_{\bf n} }\overline U_{\rho_{\bf n} , \beta_{\bf n} } \\
							\quad \quad   = F_{Q_{\bf n}}({\overline U_{\rho_{\bf n}, \beta_{\bf n}  } }({\bf x},t))+\widehat G_{\beta_{\bf n} }({\bf x},t) ,~~0<t<T, \hfill \\
							{\overline U_{\rho_{\bf n}, \beta_{\bf n}  } }( {\bf x},t)= 0,~~{\bf x} \in \partial \Omega,\hfill\\
							{\overline U_{\rho_{\bf n}, \beta_{\bf n}  } }( {\bf x},T)=\widehat H_{\beta_{\bf n}}({\bf x}), \hfill\\
						\end{gathered}  \right. \label{regu011111}
						\eq		
						where
						\[
						F_{Q_{\bf n}} \left({\overline U_{\rho_{\bf n}, \beta_{\bf n}  } } ({\bf x},t) \right)
						=
						\begin{cases}
						Q_{\bf n}-Q_{\bf n}^3, &\quad 	\overline U_{\rho_{\bf n}, \beta_{\bf n}   }(	{\bf x},t) >Q_{\bf n},\\
						\overline U_{\rho_{\bf n}, \beta_{\bf n}   }-\left(\overline U_{\rho_{\bf n}, \beta_{\bf n}   }\right)^3, &\quad - Q_{\bf n} \le 	\overline U_{\rho_{\bf n}, \beta_{\bf n}   }(	{\bf x},t)\le Q_{\bf n}, \\
						- Q_{\bf n}+Q_{\bf n}^3,&\quad 	\overline U_{\rho_{\bf n}, \beta_{\bf n}   }(	{\bf x},t) < - Q_{\bf n}.
						\end{cases}
						\]
						Assume that $H, G, {\bf u}$ be as  in Theorem  \eqref{theorem4.1}. Let   $ \bn $ be  as in \eqref{betan}. Let  $\rho_{\bf n}$ and $	Q_{\bf n}$ be such that
						\begin{equation}
							\rho_{\bf n}+ \rho_{\bf n}^2= \frac{2\al_0}{ m_1 T(4 \al_0+d) }  \log \left(  \prod_{k=1}^d  n_k \right)
						\end{equation}
						and
						\begin{equation}
							1+3	Q_{\bf n}^2 = \frac{ \delta_0-1 }{4T} \log \left( \overline \Pi({\bf n}) \right)
						\end{equation}
						respectively.  {The last two equations can be solved by a simple way that  leads to  the value of $\rho_{\bf n}$ and $Q_{\bf n}$.  }	
						Then the error between  the solution {\bf u} of Problem \eqref{Fisher-Kolmogorov} and  the solution $\overline U_{\rho_{\bf n}, \beta_{\bf n}   }$ of Problem \eqref{regu011111} is of order
						\begin{equation} \label{4.35}
							\max \Bigg[ \frac{1 }{\prod_{k=1}^d  ( n_k)^{4\mu_k\delta_0-\delta_0} }, \left( \prod_{k=1}^d  n_k  \right)^{\frac{2\al_0 \delta_0- \mu_0 \delta_0}{ 2 \al_0+d/2 }} ,
							\left( \prod_{k=1}^d  n_k  \right)^{\frac{-4\al_0 \delta \delta_0}{ 4 T\al_0+dT }}\Bigg]   \left( \prod_{k=1}^d  n_k  \right)^{\frac{-2\al_0 t}{ 2 T\al_0+dT/2 }}.
						\end{equation}
							{ It is easy to check that the term in \eqref{4.35} tends to zero as $|{\bf n}| \to +\infty$}.
					\end{corollary}

					\subsubsection{\bf Proof of the main results}

					\begin{proof}[\bf Proof of Theorem \ref{theorem4.1}]
						
						{\bf Step 1.}
						{\it The existence and uniqueness of the regularized problem.} We refer to  the proof of  Theorem \ref{theorem4.2} where  we prove existence and uniqueness   for the more general  operator $\A(t,u)$ which is more general than the  operator $\A(t)$ in  Theorem \ref{theorem4.1} and   $\A(t,u)= \A(t)$.% {\color{red} We have not proved the theorem \ref{theorem4.2} yet? do you want to say the proof is similar to the proof of theorem \ref{theorem4.2}?}\\

						\noindent 	{\bf Step 2.} {\it Regularity of the regularized solution $\overline U_{\rho_{\bf n},\bn}$. }\\
						Let us define the function
						\begin{equation}
							\widetilde U_{\rho_{\bf n}, \bn }({\bf x},t)=  e^{ \kappa_{\bf n} (t-T)} \overline U_{\rho_{\bf n}, \bn}({\bf x},t),
						\end{equation}	
						where $\kappa_{\bf n} $ is positive constant to be selected later.
						By taking the partial derivative of $\widetilde U_{\rho_{\bf n}, \bn}({\bf x},t)$ with respect to $t$, we obtain
						\begin{eqnarray}
							\begin{aligned}
								\frac{\partial \widetilde U_{\rho_{\bf n}, \bn}}{\partial t}({\bf x},t)  &= e^{\kappa_{\bf n} (t-T)} 	\frac{\partial \overline U_{\rho_n,\bn} }{\partial t}({\bf x},t) + \kappa_{\bf n} e^{\kappa_{\bf n}(t-T)}    \overline U_{\rho_{\bf n}, \bn} ({\bf x},t)\nn\\
								&= -e^{\kappa_{\bf n} (t-T)}\Big(\A(t)-\mathbb{\bf P}+{\bf P}_{\rho_{\bf n}}\Big)  \overline U_{\rho_{\bf n}, \bn}({\bf x},t)+ e^{\kappa_{\bf n}(t-T)} F_{Q_{\bf n}}({\overline U_{\rho_{\bf n}, \bn} }({\bf x},t))\nn\\
								&\quad \quad \quad +e^{\kappa_{\bf n}(t-T)} \widehat G_{\rho_{\bf n}}({\bf x},t)+  \kappa_n e^{\kappa_{\bf n}(t-T)}  \overline U_{\rho_{\bf n}, \bn}({\bf x},t) \nn\\
								&= -\Big(\A(t)-\mathbb{\bf P}+{\bf P}_{\rho_{\bf n}}\Big)\widetilde U_{\rho_{\bf n}}  ({\bf x},t)+  \kappa_{\bf n} \widetilde U_{\rho_{\bf n}, \bn}  ({\bf x},t)\nn\\
								&\quad \quad \quad +e^{\kappa_{\bf n}(t-T)} \widehat G_{\beta_{\bf n}}({\bf x},t)+e^{\kappa_{\bf n}(t-T)} F_{Q_{\bf n}}({\overline U_{\rho_{\bf n}, \bn} }({\bf x},t)).
							\end{aligned}	
						\end{eqnarray}
						Taking the inner product of both sides with $\widetilde U_{\rho_{\bf n}, \bn}  ({\bf x},t)$ gives
						\begin{align}
							&\frac{1}{2} \frac{\partial }{\partial t}\| \widetilde U_{\rho_{\bf n}, \bn}  (.,t)\|^2_{L^2(\Omega)}\nn\\
							&= \underbrace {\Big\langle \big( \mathbb{\bf P}-\A(t)\big)\widetilde U_{\rho_{\bf n}, \bn}  (.,t) , \widetilde U_{\rho_{\bf n},\bn}  (.,t)\Big \rangle_{L^2(\Omega)}}_{=: \mathcal J_1} +  \underbrace {\Big\langle  {\bf P}_{\rho_{\bf n}} \widetilde U_{\rho_{\bf n} , \bn}  (.,t), \widetilde U_{\rho_{\bf n}, \bn }  (.,t) \Big \rangle_{L^2(\Omega)} }_{=:\mathcal J_2}\nn\\
							&+ \kappa_{\bf n} \|\widetilde U_{\rho_{\bf n}, \bn}  (.,t) \|^2_{L^2(\Omega)}\nn\\
							&+ \underbrace { \Big\langle  e^{\kappa_{\bf n} (t-T)} F_{Q_{\bf n}}({\overline U_{\rho_{\bf n}, \bn} }(.,t)),\widetilde U_{\rho_{\bf n} , \bn}  (.,t)  \Big \rangle_{L^2(\Omega)} }_{=: \mathcal J_3}+\underbrace { \Big\langle  e^{\kappa_{\bf n}  (t-T)} \widehat G_{\beta_{\bf n} }(.,t),\widetilde U_{\rho_{\bf n}, \bn}  (.,t)  \Big \rangle_{L^2(\Omega)} }_{=: \mathcal J_4}.
						\end{align}
						It remains to estimate
						$\mathcal J_1, \mathcal J_2$ and $\mathcal J_3$.
						By the conditions for $\mathbb{\bf P}$ in Assumption  (A) in  Definition \ref{asumption1111}, we obtain that
						\begin{equation}
							\mathcal J_1 ~\ge~  \widehat M \Vert \widetilde U_{\rho_{\bf n}, \bn}  (.,t)\Vert_{{\mathcal V}(\Omega)}^2. \label{estimate11}
						\end{equation}
						The term $\mathcal J_2$ can be estimated as follows
						\begin{equation}
							|\mathcal J_2|  ~\le~  \|{\bf P}_{\rho_{\bf n} }\|_{\mathbb{L} (L^2(\Omega);L^2(\Omega)) } \| \widetilde U_{\rho_{\bf n}, \bn }  (.,t)\|^2_{L^2(\Omega)} ~\le~ {\overline  M} (\sqrt{\rho_{\bf n} })\| \widetilde U_{\rho_{\bf n}, \bn }  (.,t)\|^2_{L^2(\Omega)} , \label{estimate22}
						\end{equation}
						where we have used the inequality \eqref{inequality1}.
						Using Cauchy-Schwarz inequality and noting the globally Lipschitz property of the function $F_{Q_{\bf n}}$,  we deduce that
						\begin{equation}\label{estimate33}
							\begin{split}
								|\mathcal J_3| & = \Big\langle  e^{\kappa_{\bf n} (t-T)} F_{Q_{\bf n} }({\overline U_{\rho_{\bf n}, \bn} }(.,t)),\widetilde U_{\rho_{\bf n}, \bn}  (.,t)  \Big \rangle_{L^2(\Omega)} \\
								&\le \frac{1}{2} e^{2\kappa_{\bf n}(t-T)}\| F_{Q_{\bf n}}({\overline U_{\rho_{\bf n}, \bn} }(.,t))\|^2_{L^2(\Omega)} + \frac{1}{2}  \|\widetilde U_{\rho_{\bf n}, \bn}  (.,t) \|^2_{L^2(\Omega)} \\
								&  \le \frac{1}{2} e^{2\kappa_{\bf n}(t-T)}\Big( 2K(Q_{\bf n}) \| \widetilde U_{\rho_{\bf n}, \bn}  (.,t)\|_{L^2(\Omega)} +  \|F(0)\|_{L^2(\Omega)} \Big) ^2+ \frac{1}{2}  \|\widetilde U_{\rho_{\bf n}, \bn}  (.,t) \|^2_{L^2(\Omega)} \\
								& \le \Big( 2K^2Q_{\bf n}+1\Big) \| \widetilde U_{\rho_{\bf n}, \bn}  (.,t)\|_{L^2(\Omega)}^2+ e^{2\kappa_{\bf n}(t-T)}  \|F(0)\|^2_{L^2(\Omega)} .
							\end{split}
						\end{equation}
						{ where we note that we have used above the fact that  $F_{Q_{\bf n}}$ (0) = F(0)}.
						Using Cauchy-Schwartz inequality, we have the bound of $	|\mathcal J_4|$ as follows
						\begin{equation}\label{estimate44}
							\begin{split}
								|\mathcal J_4|& =  \Big|	\Big\langle  e^{\kappa_{\bf n}(t-T)} \widehat G_{\beta_{\bf n}}(.,t),\widetilde U_{\rho_{\bf n}, \bn}  (.,t)  \Big \rangle_{L^2(\Omega)} \Big| \\
								&\le  \frac{1}{2} e^{2\kappa_{\bf n}(t-T)}\| \widehat G_{\beta_{\bf n}}(.,t) \|^2_{L^2(\Omega)} + \frac{1}{2}  \|\widetilde U_{\rho_{\bf n}, \bn}  (.,t) \|^2_{L^2(\Omega)} .
							\end{split}
						\end{equation}
						Combining \eqref{estimate11}, \eqref{estimate22},\eqref{estimate33},\eqref{estimate44} we have that
						\begin{align*}
							\frac{1}{2} \frac{\partial }{\partial t}\| \widetilde U_{\rho_{\bf n}, \bn}  (.,t)\|^2_{L^2(\Omega)} &\ge \widehat M \Vert \widetilde U_{\rho_{\bf n}, \bn}  (.,t)\Vert_{{\mathcal V}(\Omega)}^2\nn\\
							&+\kappa_{\bf n} \|\widetilde U_{\rho_{\bf n}, \bn}  (.,t) \|^2_{L^2(\Omega)}- {\overline  M} (\sqrt{\rho_{\bf n}})\| \widetilde U_{\rho_{\bf n}, \bn}  (.,t)\|^2_{L^2(\Omega)}\nn\\
							&-\Big( 2K^2(Q_{\bf n})+1\Big) \| \widetilde U_{\rho_{\bf n}, \bn}  (.,t)\|_H^2+ e^{2\kappa_{\bf n}(t-T)}  \|F(0)\|^2_{L^2(\Omega)}.
						\end{align*}
						Integrating the last inequality over $[t,T]$ yields
						\begin{eqnarray*}
							\begin{aligned}
								&{\bf E}\| \widetilde U_{\rho_{\bf n}, \bn}  (.,T)\|^2_{L^2(\Omega)} +e^{2\kappa_{\bf n}(t-T)}{\bf E}\| \widehat G_{\beta_{\bf n}}(.,t) \|^2_{L^2(\Omega)} + (T-t) \|F(0)\|^2_{L^2(\Omega)}  \nn\\
								&\ge {\bf E}\| \widetilde U_{\rho_{\bf n}, \bn}  (.,t)\|^2_{L^2(\Omega)}+ 2 \widehat M {\bf E} \left[ \int_t^T  \Vert \widetilde U_{\rho_{\bf n}, \bn}  (.,\tau)\Vert_{{\mathcal V}(\Omega)}^2 d\tau\right]\nn\\
								&+\int_t^T  \Big(2 \kappa_{\bf n}-2 {\overline  M} (\sqrt{\rho_{\bf n}}) -4K^2(Q_{\bf n}) -2 \Big){\bf E} \| \widetilde U_{\rho_{\bf n}, \bn}  (.,\tau)\|^2_{L^2(\Omega)}  d\tau.
							\end{aligned}
						\end{eqnarray*}
						By  choosing  $\kappa_{\bf n}={\overline  M} (\sqrt{\rho_{\bf n}})$, we derive that
						\begin{align}
							e^{2(t-T){\overline  M} (\sqrt{\rho_{\bf n}})} {\bf E}\|  \overline U_{\rho_{\bf n}, \bn}(.,t)\|^2_{L^2(\Omega)}  	&\le \Big( 4K^2(Q_{\bf n})  +2 \Big) \int_t^T e^{2(\tau-T){\overline  M} (\sqrt{\rho_{\bf n}})} {\bf E} \| \widetilde U_{\rho_{\bf n}, \bn}  (.,\tau)\|^2_{L^2(\Omega)}  d\tau\nn\\
							&+{\bf E}\| \widehat H_{\beta_{\bf n}}\|^2_{L^2(\Omega)} +e^{2(t-T){\overline M } (\sqrt{\rho_{\bf n}})} {\bf E}\| \widehat G_{\beta_{\bf n}}(.,t) \|^2_{L^2(\Omega)} \nn\\
							&+ (T-t) \|F(0)\|^2_{L^2(\Omega)}.
							\label{estimate1}
						\end{align}
						Multiplying both sides of the last inequality by $e^{2T{\overline  M} (\sqrt{\rho_{\bf n}})} $, we get
						\begin{eqnarray*}
							\begin{aligned}
								e^{2t{\overline  M} (\sqrt{\rho_{\bf n}})} {\bf E}\|  \overline U_{\rho_{\bf n}, \bn}(.,t)\|^2_{L^2(\Omega)}  	&\le \Big( 4K^2(Q_{\bf n}) +2 \Big) \int_t^T e^{2\tau{\overline  M} (\sqrt{\rho_{\bf n}})} {\bf E} \| \widetilde U_{\rho_{\bf n}, \bn}  (.,\tau)\|^2_{L^2(\Omega)}  d\tau\nn\\
								&+e^{2T{\overline  M} (\sqrt{\rho_{\bf n}})} {\bf E}\| \widehat H_{\beta_{\bf n}}\|^2_{L^2(\Omega)} +e^{2t{\overline  M } (\sqrt{\rho_{\bf n}})} {\bf E}\| \widehat G_{\beta_{\bf n}}(.,t) \|^2_{L^2(\Omega)} \nn\\
								&+e^{2T{\overline  M } (\sqrt{\rho_{\bf n}})} (T-t) \|F(0)\|^2_{L^2(\Omega)}.
							\end{aligned}
						\end{eqnarray*}
						Noting that $$e^{2t{\overline  M} (\sqrt{\rho_{\bf n}})} {\bf E}\| \widehat G_{\beta_{\bf n}}(.,t) \|^2_{L^2(\Omega)}  \le e^{2T{\overline  M} (\sqrt{\rho_{\bf n}})} {\bf E}\| \widehat G_{\beta_{\bf n}} \|^2_{L^\infty ( 0,T; L^2(\Omega))}, $$ we deduce that
						\begin{align}
							e^{2t{\overline  M} (\sqrt{\rho_{\bf n}})} {\bf E}\|  \overline U_{\rho_{\bf n}, \bn}(.,t)\|^2_{L^2(\Omega)}  	&\le \Big( 4K^2(Q_{\bf n})   +2 \Big) \int_t^T e^{2\tau{\overline  M } (\sqrt{\rho_{\bf n}})} {\bf E} \| \widetilde U_{\rho_{\bf n}, \bn}  (.,\tau)\|^2_{L^2(\Omega)}  d\tau\nn\\
							&+e^{2T{\overline  M} (\sqrt{\rho_{\bf n}})} {\bf E}\| \widehat H_{\beta_{\bf n}}\|^2_{L^2(\Omega)} +e^{2T{\overline M} (\sqrt{\rho_{\bf n}})} {\bf E}\| \widehat G_{\beta_{\bf n}} \|^2_{L^\infty ( 0,T; L^2(\Omega))} \nn\\
							&+e^{2T{\overline  M} (\sqrt{\rho_{\bf n}})} T \|F(0)\|^2_{L^2(\Omega)}.
							\label{estimate1111}
						\end{align}
						Applying Gronwall's inequality to the last inequality, we get
						\begin{align}
							&e^{2t{\overline  M} (\sqrt{\rho_{\bf n}})} {\bf E}\|  \overline U_{\rho_{\bf n}, \bn}(.,t)\|^2_{L^2(\Omega)}   \nn\\
							&\le \exp \Big(  4K^2(Q_{\bf n}) (T-t)\Big) e^{2T{\overline  M} (\sqrt{\rho_{\bf n}})} \Big[  {\bf E}\| \widehat H_{\beta_{\bf n}}\|^2_{L^2(\Omega)} + {\bf E}\| \widehat G_{\beta_{\bf n}} \|^2_{L^\infty ( 0,T; L^2(\Omega))} + T \|F(0)\|^2_{L^2(\Omega)}\Big].
							\label{estimate1111}
						\end{align}
						Multiplying both sides of the last inequality with  $e^{-2t{\overline  M} (\sqrt{\rho_{\bf n}})} $, we get the upper bound of    $ {\bf E}\|  \overline U_{\rho_{\bf n}}(.,t)\|^2_{L^2(\Omega)}$
						{  which shows  the stability of $\overline U_{\rho_{\bf n}, \bn}(.,t)$ in the sense of  the solution $\overline U_{\rho_{\bf n}, \bn}(.,t)$  depend  continuously on     the given data $\widehat H_{\beta_{\bf n}},~\widehat G_{\beta_{\bf n}} $ and $F$ .}
					
{\bf Step 3.} {\it
							Error estimate between the regularized solution and the sought solution.}\\	
						It is easy to see that ${\bf u}$ satisfies
						\begin{align*}
							\frac{\partial  {\bf u} ({\bf x},t)}{\partial t}+ {\bf P}_{\rho_{\bf n}  } {\bf u}({\bf x},t)&= F({\bf u}({\bf x},t))+G({\bf x},t)\nn\\
							&+ \big(  {\bf P}_{\rho_{\bf n}} -\mathbb{\bf P}\big){\bf u}({\bf x},t)+\big(\mathbb{\bf P}-\A(t)\big){\bf u}({\bf x},t).
						\end{align*}
						Putting
						$$\widetilde Z_{\rho_{\bf n}, \bn }({\bf x},t) = \overline U_{\rho_{\bf n}, \bn}({\bf x},t)-{\bf u}({\bf x},t),$$ we have
						\begin{align*}
							\frac{\partial }{\partial t} \widetilde Z_{\rho_{\bf n}, \bn}({\bf x},t) +{\bf P}_{\rho_{\bf n}} \widetilde Z_{\rho_{\bf n}, \bn}({\bf x},t)& =  F_{Q_{\bf n}}  (\overline U_{\rho_n, \bn}({\bf x},t))-  F({\bf u}({\bf x},t))-\big({\bf P}_{\rho_{\bf n}} -\mathbb{\bf P}\big){\bf u}({\bf x},t)\nn\\
							&+\Big(\mathbb{\bf P}-\A(t)\Big)\widetilde Z_{\rho_{\bf n}, \bn}({\bf x},t)+\widehat G_{\beta_{\bf n}}({\bf x},t)-G({\bf x},t).
						\end{align*}
						Put $$  {\bf X}_{\rho_{\bf n}, \bn}({\bf x},t)=e^{\kappa_{\bf n}(t-T)}\widetilde Z_{\rho_{\bf n}, \bn}({\bf x},t). $$
						Take the inner product of the both sides of the last equality by ${\bf X}_{\rho_{\bf n}, \bn}({\bf x},t)$  and  then by  integrating with respect to the time variable, it follows that
						\begin{align}
							&\Vert  {\bf X}_{\rho_{\bf n}, \bn}(.,T)\Vert^2_{L^2(\Omega)}- \Vert  {\bf X}_{\rho_{\bf n}, \bn}(.,t)\Vert^2_{L^2(\Omega)}\nn\\
							& = 2\kappa_{\bf n} \int_t^T \Vert  {\bf X}_{\rho_{\bf n}, \bn}(.,\tau)\Vert^2_{L^2(\Omega)} d\tau- \underbrace{2\int_t^T \int_{\Omega}  \Big({\bf P}_{\rho_{\bf n}}   {\bf X}_{\rho_{\bf n}, \bn}({\bf x},\tau) \Big) {\bf X}_{\rho_{\bf n}, \bn}({\bf x},\tau)d {\bf x} d\tau}_{:=\mathcal J_4}\nn\\
							&+ \underbrace{
								2 e^{\kappa_{\bf n}(t-T)} \int_t^T \int_{\Omega}  F_{Q_{\bf n}}(\overline U_{\rho_{\bf n}, \bn}({\bf x},\tau))-  F({\bf u}({\bf x},\tau))  {\bf X}_{\rho_{\bf n}, \bn}({\bf x},\tau) d {\bf x} d\tau}_{:=\mathcal J_5}\nn\\
							&+\underbrace{
								2e^{\kappa_{\bf n}(t-T)}  \int_t^T \int_{\Omega} \widehat G_{\beta_{\bf n}}({\bf x},\tau)-G({\bf x},\tau)  {\bf X}_{\rho_{\bf n}, \bn}({\bf x},\tau) d{\bf x} d\tau}_{:=\mathcal J_6}\nn\\
							&+ \underbrace{2e^{\kappa_{\bf n} (t-T)} \int_t^T \int_{\Omega}  \left(\mathbb{\bf P}-\A(\tau)\right)  \widetilde Z_{\rho_{\bf n}}({\bf x},\tau)   {\bf X}_{\rho_{\bf n}, \bn}({\bf x},\tau)  d{\bf x} d\tau }_{:=\mathcal J_7}
							\nn\\
							&	+ \underbrace{2
								e^{\kappa_{\bf n} (t-T)}  \int_t^T \int_{\Omega} \left({\bf P}_{\rho_{\bf n}}-\mathbb{\bf P}\right) {\bf u} ({\bf x},\tau)  {\bf X}_{\rho_{\bf n}, \bn}({\bf x},\tau)d {\bf x} d\tau }_{:=\mathcal J_8}. \label{ess1}
						\end{align}
						By the Cauchy-Schwartz inequality, the expectation of absolute of $\mathcal J_4$ is bounded by
						\begin{align}
							{\bf E} \Big| \mathcal J_4 \Big| &\le 2 	{\bf E} \left[  \int_t^T \sqrt{\Big( \int_{\Omega} \Big|{\bf P}_{\rho_{\bf n}}   {\bf X}_{\rho_{\bf n}, \bn}({\bf x},\tau)\Big|^2 d{\bf x}\Big)  \Big(  \int_{\Omega} | {\bf X}_{\rho_{\bf n},\bn}({\bf x},\tau)|^2d {\bf x} \Big)} d\tau  \right] \nn\\
							&\le 2 {\overline  M} (\sqrt{\rho_{\bf n}}) \int_t^T   {\bf E}  \|{\bf X}_{\rho_{\bf n}, \bn}({\bf x},\tau)\|^2_{L^2(\Omega)}  d\tau. \label{ess2}
						\end{align}
						For $|{\bf n}|$ large enough, we recall that $F_{Q_{\bf n}} ({\bf u})= F({\bf u})$ and  using the global Lipschitz property  of $F_{Q_{\bf n}}$, we have
						 the bound of $\mathcal J_5$ as follows by using Cauchy-Schwartz inequality
						\begin{align}
							{\bf E} \Big| \mathcal J_5 \Big| &\le 2 	{\bf E} \left[  \int_t^T  \Big|F_{Q_{\bf n}}(\overline U_{\rho_{\bf n}, \bn}({\bf x},\tau))-  F({\bf u}({\bf x},\tau))\Big|_{L^2(\Omega)}   \Big\| {\bf X}_{\rho_{\bf n}, \bn}({\bf x},\tau)\Big\|^2  d\tau  \right] \nn\\
							&= 2 	{\bf E} \left[  \int_t^T  \Big\|F_{Q_{\bf n}}(\overline U_{\rho_{\bf n}, \bn}({\bf x},\tau))-  F_{Q_{\bf n}}({\bf u}({\bf x},\tau))\Big\|^2   \Big\| {\bf X}_{\rho_{\bf n}, \bn}({\bf x},\tau)\Big\|_{L^2(\Omega)} d\tau  \right] \nn\\
							&\le 4K(Q_{\bf n}) \int_t^T   {\bf E}  \|{\bf X}_{\rho_{\bf n}, \bn}(.,\tau)\|^2_{L^2(\Omega)}  d\tau. \label{ess3}
						\end{align}
						The term $\mathcal J_6$ is bounded  by
						\begin{align}
							{\bf E} \Big| \mathcal J_6 \Big| &\quad\le \quad 	{\bf E} \left[  \int_t^T \Big( \int_{\Omega} \Big|\widehat G_{\beta_{\bf n} }({\bf x},\tau)-G({\bf x},\tau)\Big|^2 d{\bf x}\Big) d\tau \right]  +{\bf E} \left[ \int_t^T \Big(  \int_{\Omega} | {\bf X}_{\rho_{\bf n}, \bn}({\bf x},\tau)|^2d{\bf x}\Big)d\tau  \right] \nn\\
							&\quad \le \quad  T {\bf E}  \Big\| \widehat G_{\beta_{\bf n}}(.)-G(.) \Big\|_{L^\infty(0,T;L^2(\Omega))}^2+ \int_t^T   {\bf E}  \|{\bf X}_{\rho_{\bf n}, \bn}(.,\tau)\|^2_{L^2(\Omega)}  d\tau. \label{ess4}
						\end{align}
						The term $\mathcal J_7$ is estimated using the Assumption (A) in Definition \ref{asumption1111} as follows
						\begin{eqnarray}
							\begin{aligned}
								{\bf E} \Big| \mathcal J_7 \Big|&=	{\bf E} \left[\int_t^T  \Big\langle \left(\mathbb{\bf P}-\A(\tau)\right)  {\bf X}_{\rho_{\bf n}, \bn}(.,\tau) ,  {\bf X}_{\rho_{\bf n}, \bn}(.,\tau) \Big\rangle_{L^2(\Omega)}  d\tau \right] \nn\\
								&\ge \widehat M {\bf E}  \int_t^T 	\|{\bf X}_{\rho_{\bf n}, \bn}(.,\tau)\|^2_{\mathcal{V}(\Omega)}d\tau  \label{ess5}
							\end{aligned}
						\end{eqnarray}
						and using Lemma \ref{lemma4.2}
						\begin{align}
							{\bf E} \Big| \mathcal J_8 \Big|&= 	{\bf E} \left[   \int_t^T 2
							e^{\kappa_{\bf n}(\tau-T)} \int_{\Omega} \left({\bf P}_{\rho_{\bf n}}-\mathbb{\bf P}\right) {\bf u} ({\bf x},\tau)  {\bf X}_{\rho_{\bf n}}({\bf x},\tau)d{\bf x} d\tau  \right]\nn\\
							&\le   \left[  \int_t^T \Big| {\overline M}(|{\sqrt{\rho_{\bf n}}}|)\Big|^{-2\gamma}  e^{-2T{\overline  M } (\sqrt{\rho_{\bf n}}) }  \Big\| {\bf u}(.,\tau) \Big\|^2_{\mathcal{G}_{1+\gamma}(\Omega) } d\tau \right]+ \int_t^T   {\bf E}  \|{\bf X}_{\rho_{\bf n}, \bn}(.,\tau)\|^2_{L^2(\Omega)}  d\tau\nn\\
							&\le T \Big| {\overline  M }(|{\sqrt{\rho_{\bf n}}}|)\Big|^{-2\gamma}   e^{-2T{\overline  M} (\sqrt{\rho_{\bf n}}) } \big\| {\bf u} \big\|^2_{L^\infty(0,T;\mathcal{G}_{1+\gamma}(\Omega)) }+\int_t^T   {\bf E}  \|{\bf X}_{\rho_{\bf n}, \bn}(.,\tau)\|^2_{L^2(\Omega)}  d\tau.\label{ess6}
						\end{align}
						Combining \eqref{ess1}, \eqref{ess2}, \eqref{ess3},\eqref{ess4} \eqref{ess5}, \eqref{ess6}  gives
						\begin{align}
							&{\bf E} \Vert  {\bf X}_{\rho_{\bf n}, \bn}(.,T)\Vert^2_{L^2(\Omega)}- {\bf E}  \Vert  {\bf X}_{\rho_{\bf n}, \bn}(.,t)\Vert^2_{L^2(\Omega)} \nn\\
							&\ge \Big( 2\kappa_{\bf n}- 2{\overline  M} (\sqrt{\rho_{\bf n}}) -4K(Q_{\bf n})-2 \Big)\int_t^T   {\bf E}  \|{\bf X}_{\rho_{\bf n}, \bn}(.,\tau)\|^2_{L^2(\Omega)}
							\nn\\
							&+\widehat M {\bf E}  \int_t^T 	\|{\bf X}_{\rho_{\bf n}, \bn}(.,\tau)\|^2_{ {\mathcal V}(\Omega)}d\tau - T \Big| {\overline  M}(|{\sqrt{\rho_{\bf n}}}|)\Big|^{-2\gamma}  e^{-2T{\overline  M} (\sqrt{\rho_{\bf n}}) } \Big\| {\bf u} \Big\|^2_{L^\infty(0,T;\mathcal{G}_{1+\gamma}(\Omega)) }\nn\\
							&-T {\bf E}  \Big\| \widehat G_{\beta_{\bf n}}(.)-G(.) \Big\|_{L^\infty(0,T;L^2(\Omega))}^2.
						\end{align}
						This leads to
						\begin{eqnarray}
							\begin{aligned}
								&e^{2\kappa_{\bf n}(t-T)}{\bf E}	\Big\| \overline U_{\rho_{\bf n},\bn}(.,t)-{\bf u} (.,t)\Big\|^2_{L^2(\Omega)}+\widehat M {\bf E}  \int_t^T 	\|{\bf X}_{\rho_{\bf n},\bn}(.,\tau)\|^2_{\mathcal{V}(\Omega)}d\tau \nn\\
								&+\Big( 2\kappa_{\bf n}- 2{\overline  M} (\sqrt{\rho_{\bf n}}) -4K(Q_{\bf n})-2 \Big)\int_t^T    e^{2\kappa_{\bf n}(\tau-T)}{\bf E}	\Big\| \overline U_{\rho_{\bf n},\bn}(.,\tau)-{\bf u} (.,\tau)\Big\|_{L^2(\Omega)}^2d\tau\nn\\
								&\quad \quad \quad \quad \le
								{\bf E}	\Big\| \widehat H_{\bn}-H \Big\|_{L^2(\Omega)}^2 +T {\bf E}  \Big\| \widehat G_{\beta_{\bf n}}(.)-G(.) \Big\|_{L^\infty(0,T;L^2(\Omega))}^2\nn\\
								&\quad \quad \quad \quad+T \Big| {\overline  M }(|{\sqrt{\rho_{\bf n}}}|)\Big|^{-2\gamma}  e^{-2T{\overline  M } (\sqrt{\rho_{\bf n}}) } \big\| {\bf u} \big\|^2_{L^\infty(0,T;\mathcal{G}_{1+\gamma}(\Omega)) }.
							\end{aligned}
						\end{eqnarray}
						Let us choose $\kappa_{\bf n}={\overline M} (\sqrt{\rho_{\bf n}}) $ and multiply both sides of the last inequality with $e^{2T{\overline  M} (\sqrt{\rho_{\bf n}})} $, then we conclude that
						\begin{align}
							&e^{2t{\overline  M } (\sqrt{\rho_{\bf n}})} {\bf E}	\Big\| \overline U_{\rho_{\bf n}, \bn}(.,t)-{\bf u} (.,t)\Big\|^2_{L^2(\Omega)}\nn\\
							&\quad \quad \quad  \le e^{2T{\overline  M} (\sqrt{\rho_{\bf n}})} \left(  {\bf E}	\Big\| \widehat H_{\beta_{\bf n}}-H \Big\|_{L^2(\Omega)}^2 +T {\bf E}  \Big\| \widehat G_{\beta_{\bf n}}(.)-G(.) \Big\|_{L^\infty(0,T;L^2(\Omega))}^2 \right)\nn\\
							&\quad \quad \quad+T  \Big| {\overline  M }(|{\sqrt{\rho_{\bf n}}}|)\Big|^{-2\gamma}  \Big\| {\bf u} \Big\|^2_{L^\infty(0,T;\mathcal{G}_{1+\gamma}(\Omega)) }\nn\\
							&\quad \quad \quad+\Big( 4K(Q_{\bf n})+2  \Big) \int_t^T    e^{2\tau{\overline  M} (\sqrt{\rho_{\bf n}})} {\bf E}	\Big\| \overline U_{\rho_{\bf n}, \bn}(.,\tau)-{\bf u} (.,\tau)\Big\|_{L^2(\Omega)}^2d\tau.
						\end{align}
						The Gronwall's inequality implies that
						\begin{align}
							&e^{2t{\overline  M} (\sqrt{\rho_n})} {\bf E}	\Big\| \overline U_{\rho_{\bf n}, \bn}(.,t)-{\bf u} (.,t)\Big\|^2_{L^2(\Omega)} \nn\\
							&\le e^{\Big( 4K(Q_{\bf n})+2  \Big) (T-t)  }    e^{2T{\overline  M } (\sqrt{\rho_{\bf n}})}  \left(  {\bf E}	\Big\| \widehat H_{\beta_{\bf n}}-H \Big\|_{L^2(\Omega)}^2 +T {\bf E}  \Big\| \widehat G_{\beta_{\bf n}}(.)-G(.) \Big\|_{L^\infty(0,T;L^2(\Omega))}^2 \right)   \nn\\
							&+e^{\Big( 4K(Q_{\bf n})+2  \Big) (T-t)  }  T \Big| {\overline  M}(|{\sqrt{\rho_{\bf n}}}|)\Big|^{-2\gamma}  \Big\| {\bf u} \Big\|^2_{L^\infty(0,T;\mathcal{G}_{1+\gamma}(\Omega)) }.\label{convergence1}
						\end{align}
						Multiplying both sides of \eqref{convergence1} with  $e^{-2t{\overline  M } (\sqrt{\rho_{\bf n}})} $and  thanks to  Corollary  \eqref{corollary2.1}, we conclude  that
						\eqref{error111} holds.	
						
					\end{proof}

					\subsection{General problem with  coefficients  that depend on $t$ and ${\bf u}$}
					\noindent Let $H  \in L^2(\Omega)$. In this section, we consider  the problem of constructing an ${\bf u} \in C([0,T;V(\Omega)]), {\bf u}'\in L^2(0,T;L^2(\Omega))$  such that
					
					\bq
					\left\{ \begin{gathered}
						{\bf u}_t+\A(t, {\bf u}){\bf u}  = F({\bf u}({\bf x},t))+G({\bf x},t) ,~~{\bf x} \in  \Omega, 0<t<T, \hfill \\
						{\bf u}({\bf x},t)= 0,~~{\bf x} \in \partial \Omega,\hfill\\
						{\bf u}({\bf x},T)= H({\bf x}), \quad {\bf x} \in  \Omega. \hfill\\
					\end{gathered}  \right. \label{problem012}
					\eq
					Let $0< R < \infty$. 	Define the following set
					\begin{equation}
						B_R(L^2(\Omega)):= \Big\{ w \in L^2(\Omega): \|w\|_{L^2(\Omega)} \le R  \Big\}.
					\end{equation}
					
					\begin{definition} \label{asumption2}
						The pair of operators $( \A(t,w), {\bf P} ) $  satisfies Assumption $(B)$ if  	the following conditions  hold
						\begin{enumerate}[{ \upshape(a)}]
							
							\item  For any $v\in L^2(\Omega)$, there exists a increasing function ${\overline  M} $ such as
							\begin{equation}
								{\bf P} v({\bf x})= \sum_{{\bf p} \in {\mathbb N}^d}  {\overline  M } (|{\bf p}|)\big\langle v, \psi_{\bf p}\big\rangle _ {L^2(\Omega)}\psi_{\bf p}({\bf x}).\label{asumption22}
							\end{equation}
							
							\item For $w \in 	B_R(L^2(\Omega)), u \in |mathcal{V}(\Omega), v \in \mathcal{V}(\Omega) $ then
							\begin{equation}\label{asumption00}
								\Big\langle \Big({\bf P}-\A(t,w)\Big)u,v \Big\rangle_{L^2(\Omega)} \le \widetilde  M_a \|u\|_{{\mathcal{V}(\Omega)}} \|v\|_{\mathcal{V}(\Omega)}.
							\end{equation}
							
							\item For $w \in 	B_R(L^2(\Omega)) $ and any $v\in \mathcal{V}(\Omega)$, there exists $\widehat M>0$ such that
							\begin{equation}
								\Big\langle ({\bf P}-\A(t,w))v,v \Big\rangle_{L^2(\Omega)}\geq \widehat M \Vert v\Vert_{{\mathcal V}(\Omega)}^2. \label{asumption11}
							\end{equation}
							
							\item 	There exists  $\widetilde M >0$ such that
							\begin{equation}
								\Big|\Big\langle \Big(A(t,w_1)-A(t,w_2)\Big)u,v\Big\rangle\Big|_{L^2(\Omega)}\leq \widetilde M \Vert u\Vert_{\mathcal{V}(\Omega)}\Vert v\Vert_{\mathcal{V}(\Omega)}\Vert w_1-w_2 \Vert_{L^2(\Omega)}, \label{errorA}
							\end{equation}
							for any $ u \in \mathcal{V}(\Omega), v\in \mathcal{V}(\Omega)$ and $w_1, w_2 \in B_R(L^2(\Omega))$.

						\end{enumerate}

					\end{definition}
					
					\begin{remark}\label{remark4.2}
						We give some particular equations of the model \eqref{problem012}. Let us choose $\A(t,{\bf u})= \nabla\Big(D(\mathbf{u})\nabla \mathbf u\Big)$ in the first equation of \eqref{problem012} then this equation is called a logistic reaction-diffusion equation. Here ${\bf u}$  represents the population density of species at location $x$ and time $t,$  $D({\bf u})$ is the density dependent diffusion coefficient, the notation $\nabla$  is the usual gradient operator and $F({\bf u})$ is a logistic type source term.
						\begin{itemize}
							\item  When  $F({\bf u})= a{\bf u}(1- {\bf u})$~$a>0$ , Problem \eqref{problem012} is called backward in time for {\bf Fisher-type logistic equations. } (See \cite{Monobe}).
							
							\item  When  $F({\bf u})= a{\bf u}^2 (1- {\bf u})$~$a>0$ , Problem \eqref{problem012} is called backward in time for {\bf Huxley  equation.}
							
							\item  When  $F({\bf u})= a{\bf u}^2 (1- {\bf u}) ({\bf u}- \theta_1)$~$a>0$ , Problem \eqref{problem012} is called backward in time for {\bf Fitzhugh-Nagumo   equation.}
						\end{itemize}

						Some more applications in biology of the above equations and  generalized problem can be found in \cite{Broad}.
					\end{remark}

					Using a similar method as in previous subsection, 	we  present  a  regularized problem  for Problem \eqref{problem012} as follows
					\bq
					\left\{ \begin{gathered}
						\frac{\partial \overline U_{\rho_{\bf n} , \bn}}{\partial t}+\A(t, \overline U_{\rho_{\bf n}, \bn } )\overline U_{\rho_{\bf n} , \bn} -{\bf P}\overline U_{\rho_{\bf n}, \bn } +{\bf P}_{\rho_{\bf n} }\overline U_{\rho_{\bf n} }  =\\
						\quad F_{Q_n}({\overline U_{\rho_{\bf n}, \bn} }({\bf x},t))+\widehat G_{\beta_{\bf n}}({\bf x},t) ,~~0<t<T, \hfill \\
						{\overline U_{\rho_{\bf n},\bn} }({\bf x},t)= 0,~~x \in \partial \Omega,\hfill\\
						{\overline U_{\rho_{\bf n}, \bn} }({\bf x},T)=\widehat H_{\beta_{\bf n}}({\bf x}). \hfill\\
					\end{gathered}  \right. \label{regu555}
					\eq
					
					\begin{theorem}  \label{theorem4.2}
						Let $H, G, {\bf u}, \bn, \rho_{\bf n}$ be as Theorem \ref{theorem4.1}. Then	the system \eqref{regu555} has a unique solution ${\overline U_{\rho_{\bf n}, \bn} } \in C([0,T];L^2(\Omega))\cap L^2(0,T;\mathcal{V}(\Omega))$.
						Choose $Q_{\bf n}$
						as in Theorem \ref{theorem4.1}.
						Then {   for ${\bf n}$ large enough},  the error
						\begin{equation}
							{\bf E}	\Big\| \overline U_{\rho_{\bf n}, \bn}(.,t)-{\bf u}(.,t) \Big\|_{L^2(\Omega)}^2 \quad \text{ is of order } \quad 		e^{\Big( 4K(Q_{\bf n})+2 +\frac{4 \widetilde M \widetilde M_0}{ \widehat M }   \Big) (T-t)  }  e^{-2t{\overline  M} (\sqrt{\rho_{\bf n}})} 	\overline \Pi({\bf n}) . \label{ss111111}
						\end{equation}
						Where the term  $\overline \Pi({\bf n})$ above  is defined in equation \eqref{def:pi-n}.
					\end{theorem}
					
					\begin{remark}
						One example for choices of $\bn, \rho_{\bf n}, Q_{\bf n}$ are given in Remark \eqref{corollary4.1}.
					\end{remark}
					\begin{corollary}
						Consider the following problem for {\bf Huxley equation}
						\bq
						\left\{ \begin{gathered}
							{\bf u}_t- \nabla\Big(D(\mathbf{u})\nabla \mathbf u\Big) ={\bf u}^2 (1- {\bf u}) +G({\bf x},t) ,~~{\bf x} \in  \Omega, 0<t<T, \hfill \\
							{\bf u}({\bf x},t)= 0,~~{\bf x} \in \partial \Omega,\hfill\\
							{\bf u}({\bf x},T)= H({\bf x}), \quad {\bf x} \in  \Omega \hfill\\
						\end{gathered}  \right. \label{Huxley}
						\eq
						where the density dependent diffusion coefficient $D$ satisfies that $ D_0 \le  D(w(x,t)) \le D_1$ for any $w \in L^2(\Omega)$ and $D_0, D_1$ are positive numbers.  We assume that $D $ is a globally Lipschitz function, i.e, there exists $\widetilde M \ge 0$ such that
						\begin{equation}
							\|D(w_1)-D(w_2)\|_{L^2(\Omega)} \le \widetilde M \|w_1-w_2\|_{L^2(\Omega)}.
						\end{equation}
						
						From  Definition \eqref{asumption2},   	 the operator ${\bf P}$   can be chosen  as follows
						\begin{equation} \label{P1}
							{\bf P} v := D_1 \Delta {v} = D_1 \sum_{{\bf p} \in {\mathbb N}^d  } {\bf p}^2  <v, \psi_{\bf p}> \psi_{\bf p},
						\end{equation}
						for any $v \in L^2(\Omega)$.
						It is easy to show that $(\A(t, w), {\bf P})$ as above, satisfies Assumption (B) in the sense of Definition  \eqref{asumption2}.
						We can easily  see that the eigenvalues of ${\bf P}$ are
						$\overline M ({\bf p})= D_1 {\bf p}^2 $.
						Next, we find the operator  ${\bf P}_{\rho_n}$ by truncating Fourier series in
						\eqref{P1} and we have
						\begin{equation}
							{\bf P}_{\rho_n} v:=  m_1 \sum_{ {\bf |p|} \le   \sqrt{\frac{\rho_{\bf n}} {  D_1 }  } } {\bf p}^2  <v, \psi_{\bf p}> \psi_{\bf p}.
						\end{equation}
						Thanks to \eqref{regu555}, a  regularized problem for \eqref{Huxley} is given below
						\bq
						\left\{ \begin{gathered}
							\frac{\partial \overline U_{\rho_{\bf n}, \beta_{\bf n}   }}{\partial t}- \nabla\Big(D(U_{\rho_{\bf n}, \beta_{\bf n}   } )\nabla U_{\rho_{\bf n}, \beta_{\bf n}   } \Big) -{\bf P}\overline U_{\rho_{\bf n}, \beta_{\bf n}  } +{\bf P}_{\rho_{\bf n} }\overline U_{\rho_{\bf n} , \beta_{\bf n} } \\
							\quad \quad   = F_{Q_{\bf n}}({\overline U_{\rho_{\bf n}, \beta_{\bf n}  } }({\bf x},t))+\widehat G_{\beta_{\bf n} }({\bf x},t) ,~~0<t<T, \hfill \\
							{\overline U_{\rho_{\bf n}, \beta_{\bf n}  } }( {\bf x},t)= 0,~~{\bf x} \in \partial \Omega,\hfill\\
							{\overline U_{\rho_{\bf n}, \beta_{\bf n}  } }( {\bf x},T)=\widehat H_{\beta_{\bf n}}({\bf x}), \hfill\\
						\end{gathered}  \right. \label{regu0555555}
						\eq		
						where
						\[
						F_{Q_{\bf n}} \left({\overline U_{\rho_{\bf n}, \beta_{\bf n}  } } ({\bf x},t) \right)
						=
						\begin{cases}
						Q_{\bf n}^2-Q_{\bf n}^3, &\quad 	\overline U_{\rho_{\bf n}, \beta_{\bf n}   }(	{\bf x},t) >Q_{\bf n},\\
						\left(\overline U_{\rho_{\bf n}, \beta_{\bf n}   }\right)^2-\left(\overline U_{\rho_{\bf n}, \beta_{\bf n}   }\right)^3, &\quad - Q_{\bf n} \le 	\overline U_{\rho_{\bf n}, \beta_{\bf n}   }(	{\bf x},t)\le Q_{\bf n}, \\
						- Q_{\bf n}^2+Q_{\bf n}^3,&\quad 	\overline U_{\rho_{\bf n}, \beta_{\bf n}   }(	{\bf x},t) < - Q_{\bf n}.
						\end{cases}
						\]

					\end{corollary}
					
					Assume that $H, G, {\bf u}$ are  as in Theorem  \eqref{theorem4.1}. Let   $ \bn $ be  as in \eqref{betan}. Let $\rho_{\bf n}$ and $	Q_{\bf n}$  be such that
					\begin{equation}
						\rho_{\bf n}= \sqrt{\frac{2\al_0}{ m_1 T(4 \al_0+d) }  \log \left(  \prod_{k=1}^d  n_k \right)}.
					\end{equation}
					and
					\begin{equation}
						2Q_{\bf n}+3	Q_{\bf n}^2 = \frac{ \delta_0-1 }{4T} \log \left( \overline \Pi({\bf n}) \right)
					\end{equation}
					respectively.		Then the error between the solution {\bf u} of Problem \eqref{Huxley} and  the solution $\overline U_{\rho_{\bf n}, \beta_{\bf n}   }$ of Problem \eqref{regu0555555} is of order
					\begin{equation} \label{4.68}
						\max \Bigg[ \frac{1 }{\prod_{k=1}^d  ( n_k)^{4\mu_k\delta_0-\delta_0} }, \left( \prod_{k=1}^d  n_k  \right)^{\frac{2\al_0 \delta_0- \mu_0 \delta_0}{ 2 \al_0+d/2 }} ,
						\left( \prod_{k=1}^d  n_k  \right)^{\frac{-4\al_0 \delta \delta_0}{ 4 T\al_0+dT }}\Bigg]   \left( \prod_{k=1}^d  n_k  \right)^{\frac{-2\al_0 t}{ 2 T\al_0+dT/2 }}.
					\end{equation}
					{ It is easy to check that the term in \eqref{4.68} tends to zero as $|{\bf n}| \to +\infty$}.
					
					\begin{proof}[\bf Proof of Theorem \ref{theorem4.2}]
						The proof is divided into some steps.\\
						{\bf Step 1.} {\it The existence and uniqueness of the solution to problem \eqref{regu555}.}
						Let $\overline V_{\rho_{\bf n}, \bn}({\bf x},t) = \overline U_{\rho_{\bf n},\bn}({\bf x},T-t)$ and define the following operator $\mathcal {B} (t,w)= {\bf P}- \mathcal {A} (t,w)$. Then it is obvious that  $\overline V_{\rho_{\bf n}, \bn}({\bf x},t) $ satisfies the following equation
						\bq
						\left\{ \begin{gathered}
							\frac{\partial \overline V_{\rho_{\bf n}, \bn}}{\partial t}+\mathcal{B}(t, \overline V_{\rho_{\bf n}, \bn} )\overline V_{\rho_{\bf n},\bn} -{\bf P}\overline U_{\rho_{\bf n},\bn}\\ \quad \quad ={\bf P}_{\rho_{\bf n}}\overline V_{\rho_{\bf n}, \bn}  - F_{Q_{\bf n}}({\overline V_{\rho_{\bf n},\bn} }({\bf x},t))-\widehat G_{\beta_{\bf n}}({\bf x},t) ,~~0<t<T, \hfill \\
							{\overline V_{\rho_{\bf n}, \bn} }({\bf x},t)= 0,~~{\bf x} \in \partial \Omega,\hfill\\
							{\overline V_{\rho_{\bf n},\bn} }({\bf x},0)=\widehat H_{\beta_{\bf n}}({\bf x}). \hfill\\
						\end{gathered}  \right. \label{regu222}
						\eq
						By the assumptions on  $\A$  above  in \eqref{asumption22}, \eqref{asumption00}, \eqref{asumption11},  \eqref{errorA}, it is easy to show that
						\begin{enumerate}[{ \upshape(i)}]
							\item For $w \in 	B_R(L^2(\Omega)), u \in \mathcal{V}(\Omega), v \in \mathcal{V}(\Omega) $ then
							\begin{equation}
								\Big\langle (  \mathcal{B}(t,w) u,v \Big\rangle_{L^2(\Omega)} \le \widetilde  M_a \|u\|_{\mathcal{V}(\Omega)} \|v\|_{\mathcal{V}(\Omega)}. \label{asumption00b}
							\end{equation}
							
							\item For any $v\in \mathcal{V}(\Omega)$ then
							\begin{equation}
								\Big\langle (  \mathcal{B}(t,w)v,v \Big\rangle_{L^2(\Omega)}\geq \widehat M \Vert v\Vert_{{\mathcal V}(\Omega)}^2. \label{asumption11b}
							\end{equation}

							\item 	There exist a subspace $ \mathcal{V}_1(\Omega) \subset L^2(\Omega)$ in which
							\begin{equation}
								\Big|\Big\langle \Big(\mathcal{B}(t,w_1)-\mathcal{B}(t,w_2)\Big)u,v\Big\rangle\Big|_{L^2(\Omega)}\leq \widetilde M \Vert u\Vert_{\mathcal{V}(\Omega)}\Vert v\Vert_{\mathcal{V}(\Omega)}\Vert w_1-w_2 \Vert_{L^2(\Omega)}, \label{errorAb}
							\end{equation}
							for any $ u \in \mathcal{V}(\Omega), v\in \mathcal{V}(\Omega)$ and $w_1, w_2 \in B_R(L^2(\Omega))$.	
						\end{enumerate}
						Hence, the  conditions  proved above for the operator $\mathcal B$ show that $\mathcal B$ satisfies the assumptions of Theorem 5.10 in   \cite{Yagi} (page 252). With the help of  Theorem 5.10 in of  \cite{Yagi}, we conclude that  the Problem \eqref{regu222} has unique solution $\overline V_{\rho_{\bf n} , \bn} \in  C([0,T]; L^2(\Omega)) \cap L^2(0,T; {\mathcal V}(\Omega)) $.

						{\bf Step 2.} {\it The error estimate for  ${\bf E}	\Big\| \overline U_{\rho_n, \bn}(.,t)-{\bf u} (.,t)\Big\|_{L^2(\Omega)}$.}\\
						First, we  have
						\begin{align*}
							\frac{\partial  {\bf u} ({\bf x}, t)}{\partial t}+ {\bf P}_{\rho_{\bf n}} {\bf u}({\bf x},t)&= F({\bf u}({\bf x},t))+G({\bf x},t)\nn\\
							&+ \big(  {\bf P}_{\rho_{\bf n}} -\mathbb{\bf P}\big){\bf u}({\bf x},t)+\big(\mathbb{\bf P}-\A(t,{\bf u})\big){\bf u}({\bf x},t),
						\end{align*}
						and
						\begin{align}
							\frac{\partial \overline U_{\rho_{\bf n}, \bn}}{\partial t}+{\bf P}_{\rho_{\bf n}}\overline U_{\rho_{\bf n},\bn}({\bf x},t) &= F_{Q_{\bf n}}({\overline U_{\rho_{\bf n},\bn} }({\bf x},t))+\widehat G_{\beta_{\bf n}}({\bf x},t)\nn\\
							&+\big(  {\bf P}_{\rho_{\bf n}} -\mathbb{\bf P}\big)\overline U_{\rho_{\bf n},\bn}+ \big(\mathbb{\bf P}-\A(t, \overline U_{\rho_{\bf n},\bn} )\big)\overline U_{\rho_{\bf n},\bn}.
						\end{align}
						Putting $\overline Z_{\rho_{\bf n},\bn}({\bf x},t) = \overline U_{\rho_{\bf n},\bn}({\bf x},t)-{\bf u}({\bf x},t)$,  we have
						\begin{align*}
							\frac{\partial }{\partial t} \overline Z_{\rho_{\bf n},\bn}({\bf x},t) &+{\bf P}_{\rho_{\bf n}} \overline Z_{\rho_{\bf n},\bn}({\bf x},t) \nn\\
							&=  F_{Q_{\bf n}}(\overline U_{\rho_{\bf n},\bn}({\bf x},t))-  F({\bf u}({\bf x},t))-\big({\bf P}_{\rho_{\bf n}} -\mathbb{\bf P}\big){\bf u}({\bf x},t)\nn\\
							&+\big(\mathbb{\bf P}-\A(t, \overline U_{\rho_{\bf n},\bn} )\big)\overline Z_{\rho_{\bf n},\bn}({\bf x},t)+\widehat G_{\beta_{\bf n}}({\bf x},t)-G({\bf x},t)\nn\\
							&+ \Big( \A(t,{\bf u})- \A(t, \overline U_{\rho_{\bf n},\bn} \Big){\bf u}({\bf x},t) .
						\end{align*}	
						Put $ \overline  {\bf X}_{\rho_{\bf n}, \bn}({\bf x},t)=e^{\kappa_{\bf n}(t-T)}\overline Z_{\rho_{\bf n}, \bn}({\bf x},t)$, and taking the inner product of the last equality by $\overline  {\bf X}_{\rho_{\bf n}, \bn}({\bf x},t)$, and integrating over $(t,T)$ we have
						\begin{align}
							&\Vert \overline  {\bf X}_{\rho_{\bf n}, \bn}({\bf x},T)\Vert^2_{L^2(\Omega)}- \Vert  \overline{\bf X}_{\rho_{\bf n}, \bn}({\bf x},t)\Vert^2_{L^2(\Omega)}\nn\\
							& = 2\kappa_{\bf n} \int_t^T \Vert  \overline {\bf X}_{\rho_{\bf n}, \bn}({\bf x},\tau)\Vert^2_{L^2(\Omega)} d\tau- \underbrace{2\int_t^T \int_{\Omega} {\bf P}_{\rho_{\bf n}, \bn}  \overline {\bf X}_{\rho_{\bf n}, \bn}({\bf x},\tau)  \overline{\bf X}_{\rho_{\bf n},\bn}({\bf x},\tau)d{\bf x} d\tau}_{:=\mathcal J_{4,4}}\nn\\
							&+ \underbrace{
								2 e^{\kappa_{\bf n}(t-T)} \int_t^T \int_{\Omega}  F_{Q_{\bf n}}(\overline U_{\rho_{\bf n}, \bn}({\bf x},\tau))-  F({\bf u}({\bf x},\tau)) \overline {\bf X}_{\rho_{\bf n}, \bn}({\bf x},\tau) d{\bf x} d\tau}_{:=\mathcal J_{5,5}}\nn\\
							&+\underbrace{
								2e^{\kappa_{\bf n} (t-T)}  \int_t^T \int_{\Omega} \widehat G_{\beta_{\bf n}}({\bf x},\tau)-G({\bf x},\tau)  \overline {\bf X}_{\rho_{\bf n}, \bn}({\bf x},\tau) d{\bf x} d\tau}_{:=\mathcal J_{6,6}}\nn\\
							&+ \underbrace{2e^{\kappa_{\bf n}(t-T)} \int_t^T \int_{\Omega}  \big(\mathbb{\bf P}-\A(t, \overline U_{\rho_{\bf n}, \bn} )\big)  \overline Z_{\rho_{\bf n}, \bn}({\bf x},\tau)   \overline{\bf X}_{\rho_{\bf n}}({\bf x},\tau)  d{\bf x} d\tau }_{:=\mathcal J_{7,7}}\nn\\
							&+ \underbrace{2
								e^{\kappa_{\bf n}(t-T)}  \int_t^T \int_{\Omega} \left({\bf P}_{\rho_{\bf n}}-\mathbb{\bf P}\right)u({\bf x},\tau) \overline {\bf X}_{\rho_{\bf n}, \bn}({\bf x},\tau)d{\bf x} d\tau }_{:=\mathcal J_{8,8}}\nn\\
							&+	\underbrace{2
								\int_t^T \int_{\Omega}	e^{\kappa_{\bf n}(\tau-T)}  \Big( \A(\tau,{\bf u})- \A(\tau, \overline U_{\rho_{\bf n}, \bn} \Big){\bf u}({\bf x},\tau) \overline {\bf X}_{\rho_{\bf n} , \bn}({\bf x},\tau)d{\bf x} d\tau }_{:=\mathcal J_{9,9}}. \label{esss1}
						\end{align}
						By a similar techniques  as in equation \eqref{ess2}, we obtain
						\begin{align}
							{\bf E} \Big( \mathcal J_{4,4} \Big) \ge  -2 {\overline  M } (\sqrt{\rho_{\bf n}}) \int_t^T   {\bf E}  \|\overline {\bf X}_{\rho_{\bf n}, \bn }(.,\tau)\|^2_{L^2(\Omega)} d\tau. \label{esss2}
						\end{align}
						By a similar techniques  as in equation \eqref{ess3}, we obtain
						\begin{align}
							{\bf E} \Big( \mathcal J_{5,5} \Big) \ge  -4K(Q_{\bf n}) \int_t^T   {\bf E}  \|\overline {\bf X}_{\rho_{\bf n}, \bn}(.,\tau)\|^2_{L^2(\Omega)} d\tau. \label{esss3}
						\end{align}
						By a similar techniques as in equation \eqref{ess4}, we obtain
						\begin{equation}
							{\bf E} \Big( \mathcal J_{6,6} \Big) \ge   -T {\bf E}  \Big\| \widehat G_{\beta_{\bf n}}(.)-G(.) \Big\|_{L^\infty(0,T;L^2(\Omega))}^2- \int_t^T   {\bf E}  \|\overline {\bf X}_{\rho_{\bf n}, \bn}(.,\tau)\|^2_{L^2(\Omega)} d\tau. \label{esss4}
						\end{equation}
						By a similar techniques as in  equation \eqref{ess5}, we obtain
						\begin{align}
							{\bf E} \Big( \mathcal J_{7,7}) \Big|&=	{\bf E} \left[\int_t^T  \Big\langle \left(\mathbb{\bf P}-\A(\tau, \overline U(\rho_{\bf n}, \bn))\right)  {\bf X}_{\rho_{\bf n}}(.,\tau),   {\bf X}_{\rho_{\bf n}, \bn}(.,\tau) \Big\rangle_{L^2(\Omega)} d\tau \right]\nn\\
							& \ge \widehat M {\bf E}  \int_t^T 	\|\overline {\bf X}_{\rho_{\bf n}, \bn}(.,\tau)\|^2_{{\mathcal{V}(\Omega)}}d\tau.  \label{esss5}
						\end{align}
						By a similar techniques as in equation  \eqref{ess6}, we obtain
						\begin{equation}
							{\bf E} \Big( \mathcal J_{8,8} \Big) \ge -T e^{-2T{\overline  M } (\sqrt{\rho_{\bf n}}) } \Big\| {\bf u} \Big\|^2_{L^\infty(0,T;\mathcal{G}_{1+\gamma}(\Omega)) }-\int_t^T   {\bf E}  \|\overline {\bf X}_{\rho_{\bf n}, \bn}(.,\tau)\|^2_{L^2(\Omega)} d\tau.\label{esss6}
						\end{equation}
						Now, we turn to estimate $J_{9,9}$. For $\tau \in (0,T)$, using \eqref{errorA}, we get
						\begin{eqnarray}
							\begin{aligned}
								&\int_{\Omega} 	e^{\kappa_{\bf n}(\tau-T)}  \Big( \A(\tau,{\bf u})- \A(\tau, \overline U_{\rho_{\bf n}, \bn} \Big){\bf u}({\bf x},\tau) \overline {\bf X}_{\rho_{\bf n}, \bn}({\bf x},\tau) d{\bf x} \nn\\
								& \le  	  \widetilde M \Vert {\bf u}(.,\tau)\Vert_{\mathcal{V}(\Omega)}\Vert {\bf X}_{\rho_{\bf n}, \bn}({\bf x},\tau)\Vert_{\mathcal{V}(\Omega)}\Vert {\bf X}_{\rho_{\bf n}, \bn}({\bf x},\tau)\Vert_{L^2(\Omega)} \nn\\
								& \le  		  \widetilde M \Vert {\bf u}\Vert_{L^\infty(0,T;\mathcal{V}(\Omega))}\Vert {\bf X}_{\rho_{\bf n}, \bn}({\bf x},\tau)\Vert_{\mathcal{V}(\Omega)}\Vert {\bf X}_{\rho_{\bf n}, \bn}({\bf x},\tau)\Vert_{L^2(\Omega)} \nn\\
								& \le  	\widetilde M \widetilde M_0 \Vert {\bf X}_{\rho_{\bf n}, \bn}({\bf x},\tau)\Vert_{\mathcal{V}(\Omega)}\Vert {\bf X}_{\rho_{\bf n}, \bn}({\bf x},\tau)\Vert_{L^2(\Omega)} \nn\\
								& \le \ 	\widetilde M \widetilde M_0 \left(  \frac{ \widehat M\Vert {\bf X}_{\rho_{\bf n}, \bn}({\bf x},\tau)\Vert_{\mathcal{V}(\Omega)}^2 } { 2\widetilde M \widetilde M_0  }  +\frac{2\widetilde M \widetilde M_0 \Vert {\bf X}_{\rho_{\bf n}, \bn}({\bf x},\tau)\Vert_{L^2(\Omega)}  } {\widehat M} \right)
							\end{aligned}
						\end{eqnarray}
						where $\widetilde M_0$ is a postive  constant such that $\widetilde M_0 \ge \Vert {\bf u}\Vert_{L^\infty(0,T;\mathcal{V}(\Omega))}$.
						This implies that
						\begin{align}
							{\bf E} \Big| \mathcal J_{9,9}|&= {\bf E} \left(  2
							\int_t^T \int_{\Omega}	e^{\kappa_{\bf n}(\tau-T)}  \Big( \A(\tau,{\bf u})- \A(\tau, \overline U_{\rho_{\bf n}, \bn} \Big){\bf u}({\bf x},\tau) \overline {\bf X}_{\rho_{\bf n}, \bn}({\bf x},\tau)d{\bf x} d\tau\right) \nn\\
							&\le \widehat M  {\bf E} \left(   \int_t^T 	\|\overline {\bf X}_{\rho_n, \bn}(.,\tau)\|^2_{{\mathcal{V}(\Omega)}}d\tau   \right)+ \frac{4 \widetilde M \widetilde M_0}{ \widehat M } {\bf E} \left(   \int_t^T 	\|\overline {\bf X}_{\rho_{\bf n}, \bn}(.,\tau)\|^2_{L^2(\Omega)}d\tau   \right). \label{esss7}
						\end{align}
						Combining \eqref{esss1}, \eqref{esss2}, \eqref{esss3},\eqref{esss4} \eqref{esss5}, \eqref{esss6}, \eqref{esss7}  gives
						\begin{eqnarray}
							\begin{aligned}
								&{\bf E} \Vert  \overline{\bf X}_{\rho_{\bf n}, \bn }(.,T)\Vert^2_{L^2(\Omega)}- {\bf E}  \Vert  \overline{\bf X}_{\rho_{\bf n}, \bn }(.,t)\Vert^2_{L^2(\Omega)} \nn\\
								&\quad  \ge \Big( 2\kappa_n- 2{\overline  M} (\sqrt{\rho_{\bf n}}) -4K(Q_{\bf n})-2-\frac{4 \widetilde M \widetilde M_0}{ \widehat M }  \Big)\int_t^T   {\bf E}  \|\overline{\bf X}_{\rho_{\bf n}, \bn}(.,\tau)\|^2_{L^2(\Omega)}
								\nn\\
								&\quad  - T e^{-2T{\overline  M} (\sqrt{\rho_{\bf n}}) } \Big\| {\bf u} \Big\|^2_{L^\infty(0,T;\mathcal{G}_{1+\gamma}(\Omega)) }-T {\bf E}  \Big\| \widehat G_{\bn}(.)-G(.) \Big\|_{L^\infty(0,T;L^2(\Omega)}^2.
							\end{aligned}	
						\end{eqnarray}		
						Letting $ \kappa_n={\overline  M} (\sqrt{\rho_{\bf n}})$ and by using Gronwall's inequality   as in the proof of  Theorem \ref{theorem3.1}, we conclude that
						\begin{align}
							&e^{2t{\overline  M} (\sqrt{\rho_{\bf n}})} {\bf E}	\Big\| \overline U_{\rho_{\bf n}, \bn}(.,t)-{\bf u} (.,t)\Big\|^2_{L^2(\Omega)}
							\nn\\
							&\le e^{\Big( 4K(Q_{\bf n})+2+\frac{4 \widetilde M \widetilde M_0}{ \widehat M }   \Big) (T-t)  }e^{2T{\overline  M } (\sqrt{\rho_{\bf n}})} {\bf E}	\Big\| \widehat H_{\beta_{\bf n} }-H \Big\|_{L^2(\Omega)}^2\nn\\
							&+ e^{\Big( 4K(Q_{\bf n})+2+\frac{4 \widetilde M \widetilde M_0}{ \widehat M }   \Big) (T-t)  }e^{2T{\overline  M } (\sqrt{\rho_{\bf n}})}  T {\bf E}  \Big\| \widehat G_{\beta_{\bf n}}(.)-G(.) \Big\|_{L^\infty(0,T;L^2(\Omega))}^2\nn\\
							&+e^{\Big( 4K(Q_{\bf n})+2+\frac{4 \widetilde M \widetilde M_0}{ \widehat M }   \Big) (T-t)  }   T  \Big\| {\bf u} \Big\|^2_{L^\infty(0,T;\mathcal{G}_{1+\gamma}(\Omega)) }.
						\end{align}
						Multiplying both sides of the last inequality  with  $e^{-2t{\overline  M } (\sqrt{\rho_{\bf n}})} $and combining  with the results in  Corollary \eqref{corollary2.1}, we get the desired result \eqref{ss111111}.
					\end{proof}

				\end{document}